\documentclass[dvips,preprint]{imsart}
\usepackage[OT1]{fontenc}
\usepackage{amsthm,amsmath}
\usepackage[colorlinks,citecolor=blue,urlcolor=blue]{hyperref}
\usepackage{amsmath,amssymb,amsthm}
\usepackage{mathrsfs}
\usepackage[english]{babel}
\usepackage[numbers]{natbib}
\usepackage[dvips]{graphicx}
\usepackage{color}
\usepackage{indentfirst}
\usepackage{multirow}
\usepackage{float}
\usepackage{subfig}

\startlocaldefs
\numberwithin{equation}{section}
\theoremstyle{plain}
\newtheorem{theorem}{Theorem}[section]
\newtheorem{lemma}{Lemma}[section]
\newtheorem{proposition}{Proposition}[section]
\newtheorem{remark}{Remark}[section]
\def\@journal{Submitted}
\endlocaldefs

\begin{document}

\newcommand\E{\mathbb{E}}
\newcommand\cov{\mathop{\text{Cov}}}
\newcommand\tr{\mathop{\text{tr}}}
\newcommand\topii{{\frac1{2\pi i}}}
\newcommand\mbar{\underline{m}}
\newcommand\reu[1]{{\color{black}#1}}
\begin{frontmatter}
  \title{On the sphericity test with  large-dimensional observations}
  \runtitle{Large-dimensional sphericity test}

 \thankstext{T1}{Research of this author was partly supported by
 the National Natural Science Foundation of China (Grant No. 11071213), the Natural Science
Foundation of Zhejiang Province (No. R6090034), and the Doctoral Program Fund of
Ministry of Education (No. J20110031).}
\thankstext{t2}{Research of this author was partly supported by
 a HKU start-up grant.}

  \begin{aug}
    \author{\fnms{Qinwen} \snm{Wang}\thanksref{T1}\ead[label=e1]{wqw8813@gmail.com}}
    \and
    \author{\fnms{Jianfeng} \snm{Yao}\thanksref{t2}\ead[label=e3]{jeffyao@hku.hk}}

    \runauthor{Q. Wang and J.Yao}

    \affiliation{Zhejiang  University and The University of Hong Kong}

    \address{Qinwen Wang  \\
      Department of Mathematics\\
      Zhejiang University \\
      \printead{e1}
    }

    \address{Jianfeng Yao \\
      Department of Statistics and Actuarial Science\\
      The University of Hong Kong\\
      Pokfulam, \quad
      Hong Kong \\
      \printead{e3}
    }
  \end{aug}

  \begin{abstract}
    In this paper, we propose corrections to the likelihood ratio test and John's test for sphericity in large-dimensions.
    New formulas for the limiting parameters in the CLT for linear spectral statistics of sample covariance matrices with general fourth moments are first established.
    Using these formulas, we derive the asymptotic
    distribution of the two proposed test statistics under the null. These asymptotics are valid for general population, i.e. not necessarily Gaussian, provided a finite fourth-moment. Extensive Monte-Carlo experiments are conducted to assess the quality of these tests with a comparison to several existing methods  from the literature.
    Moreover, we also obtain their asymptotic power functions  under the alternative of a  spiked population model as  a specific alternative.
  \end{abstract}



\tableofcontents
\end{frontmatter}
\section{Introduction}

Consider a sample $Y_1,\ldots,Y_n$
from a $p$-dimensional multivariate distribution
with covariance matrix $\Sigma_p$.
An important problem in multivariate  analysis is to test the
sphericity, namely
the hypothesis  $H_{0}: \Sigma_p = \sigma^{2}I_p$
where $\sigma^{2}$ is {\em unspecified}.
If the observations represent a multivariate error
with $p$ components, the null hypothesis expresses the fact
that the error is cross-sectionally uncorrelated (independent if in
addition  they are  normal)
and  have a same variance
(homoscedasticity).

Much of the  existing
theory about  this
test  has been exposed first in  details in \cite{r17}
 about Gaussian likelihood ratio test and later in \cite{r10}, \cite{r11}, \cite{r27} and also in textbooks like \cite[Chapter 8]{r18} and
\cite[Chapter 10]{r1}.
Assume for a   moment that the
sample has a normal distribution with mean zero and covariance matrix $\Sigma_p$. Let
  $S_n =n^{-1} \sum_i Y_iY_i^*$ be the sample covariance matrix and
denote its eigenvalues by $\{\ell_i\}_{1\le i\le p}$.
Two well established procedures for testing the sphericity  are the
likelihood ratio test ({\em{LRT}}) and a test devised in \cite{r10}.
The likelihood ratio statistic is,
see e.g. \cite[\S 10.7.2]{r1},
\begin{equation*}
  \label{Ln}
  L_n =   \left(  \frac{ (\ell_1\cdots \ell_p)^{1/p}} {\frac1p (\ell_1+\cdots+\ell_p)}
         \right)^{\frac12pn}~,
\end{equation*}
which is a power of the ratio of the geometric mean of the sample
eigenvalues to the arithmetic mean.  It is here noticed that in this formula
it is necessary to assume that  $p\le n$ to avoid null eigenvalues in
(the numerator of) $L_n$.  If we let $n\to\infty$ while keeping $p$ fixed, classical
asymptotic theory indicates that under the
null hypothesis,
$-2\log L_n  \Longrightarrow \chi^{2}_f $, a chi-square distribution
with degree of freedom  $f= \frac{1}{2}p(p+1)-1$.
This asymptotic distribution is further refined by
the following Box-Bartlett correction (referred as {\em{BBLRT}}):
\begin{eqnarray}\label{exp-Ln}
  P(-2\rho\log L_n\le x) = P_f ( x) +
  \omega_2\left\{ P_{f+4}  (x ) - P_f(x)   \right\} +O(n^{-3})~,
\end{eqnarray}
where  $P_k(x)= P(\chi^2_k \le x)$ and
\[     \rho= 1- \frac{2p^2+p+2}{6pn},\qquad
\omega_2=\frac{ (p+2)(p-1)(p-2)(2p^3+6p^2+3p+2)}{288p^2n^2\rho^2}~.
\]

By  observing that the asymptotic
variance of  $-2\log L_n$ is proportional to  $\tr \{
\Sigma(\tr \Sigma)^{-1}-p^{-1}I_p \}^2$,  \cite{r10} proposed to use the
statistic
\begin{equation*}
  T_2 = \frac{ p^2n}{2} \tr\left\{ S_n(\tr S_n)^{-1} - p^{-1}I_p
    \right\}^2~
    \label{Nagao}
\end{equation*}
for testing sphericity. \reu{When $p$ is fixed and $n\to\infty$, under the
null hypothesis, it also holds that
$T_2  \Longrightarrow \chi^{2}_f $, which we referred to as \em{John's test}.}
It is observed that $T_2$ is proportional to the square of the
coefficient of  variation of  the sample
eigenvalues, namely
\[  T_2 = \frac{np}{2} \cdot
\frac{ p^{-1} \sum (\ell_i -\overline \ell)^2}{{\overline \ell}^2}~,
\qquad \text{with~~} \overline \ell = \frac1n\sum_i \ell_i~.
\]
Following the idea of the Box-Bartlett correction,  \cite{r19} established
an  expansion for the distribution function of the statistics
$T_2$ (referred as {\em{Nagao's test}}),
\begin{eqnarray}
  P(T_2\le x) &=& P_f(x) + \frac1n \left\{a_p P_{f+6} (x) + b_p P_{f+4}(x)+c_p
  P_{f+2}(x) +d_p P_f(x)
  \right\}\nonumber \\
  &&+ O(n^{-2}),
\label{exp-Nagao}
\end{eqnarray}
where
\begin{eqnarray*}
  && a_p =  \frac1{12}(p^3+3p^2-12-200p^{-1}),~
  b_p =
  \frac1{8}(-2p^3-5p^2+7p-12-420p^{-1})~,
  \\
  &&
  c_p =  \frac1{4}(p^3+2p^2-p-2-216p^{-1})~,
  d_p =  \frac1{24}(-2p^3-3p^2+p+436p^{-1})~.
\end{eqnarray*}

It has been well known that classical multivariate procedures are in general challenged
by large-dimensional data. A small simulation experiment is conducted to explore the performance of the BBLRT and Nagao's test (two corrections) with growing dimension $p$.
The sample size is set to $n=64$ while dimension $p$ increases from
4 to 60 (we have also run other experiments with larger sample sizes
$n$ but conclusions are very similar),
and the nominal level is set to be  $\alpha=0.05$. The samples come from normal vectors with mean zero and identity covariance matrix, and  each pair of $(p,n)$ is assessed with 10000 independent replications.

Table \ref{tradsize} gives  the empirical sizes of BBLRT and Nagao's test. It is found here that when the dimension to sample size ratio $p/n$ is below $1/2$, both tests have an empirical size close to the nominal test level $0.05$. Then when  the ratio grows up, the BBLRT becomes quickly biased while Nagao's test still has a correct empirical size.
It is striking that although Nagao's test is derived under classical ``$p$ fixed, $n\rightarrow \infty$'' regime, it is remarkably robust against dimension inflation.

\begin{table}[htbp]
\centering
\caption{\label{tradsize} Empirical sizes of BBLRT and  Nagao's test
  at $5\%$ significance level based on $10000$ independent
  replications using normal vectors
  $\mathcal{N}(0,I_p)$ for $n=64$ and different values of $p$.}
\begin{tabular}{cccccccc}
\hline
$(p,n)$ &
(4,64) &
(8,64) &
(16,64)&
(32,64)&
(48,64)&
(56,64)&
(60,64)
\\  \hline
BBLRT  &
0.0483 &
0.0523 &
0.0491 &
0.0554 &
0.1262 &
0.3989 &
0.7605
\\
Nagao's test &
 0.0485 &
 0.0495 &
 0.0478 &
 0.0518 &
 0.0518 &
 0.0513 &
 0.0495
\\ \hline
\end{tabular}
\end{table}

Therefore, the goal of this paper is  to propose  novel corrections to both LRT and John's
test to cope with  the large-dimensional context.
Similar works have already been
done in \cite{r15},  which confirms the robustness of John's test in large-dimensions; however, these results assume a Gaussian population. In this paper,  we remove such a  Gaussian restriction, and prove that the robustness of John's test is in fact general. Following the idea of \cite{r15}, \cite{r9} proposed to use a family of well selected U-statistics to test the sphericity; however, as showed in our simulation study in Section \ref{sec:simul}, the powers of our corrected John's test are slightly higher than this test in most cases. More recently, \cite{r26} examined the performance of $T_1$ (a statistic first put forward in \cite{r24}) under non-normality, but with the moment condition $\gamma=3+O(p^{-\epsilon})$, which essentially matches the Gaussian case ($\gamma=3$) asymptotically. We have also removed this moment restriction in our setting. In short, we have unveiled two corrections that have a better performance and removed the  Gaussian or nearly Gaussian restriction found in the existing literature.

From the technical point of view, our approach differs
from \cite{r15} and
follows the one devised in \cite{r4} and \cite{r6}.
The central tool  is a CLT for linear spectral statistics
of sample covariance matrices
established in \cite{r2} and later refined in
\cite{r21}. The paper also contains an original
contribution on this CLT reported in the Appendix:
new formulas for the limiting
parameters in  the CLT.
Since such CLT's are increasingly important in large-dimensional
statistics, we believe that these new formulas  will be of
independent interest for applications other than those considered
in this paper.

The remaining of the  paper is organized as follows.
Large-dimensional corrections to  LRT and John's test are
introduced in   Section~\ref{sec:corrected}.
Section~\ref{sec:simul} reports a detailed Monte-Carlo
study to analyze finite-sample sizes and powers of these two
corrections under both normal  and
non-normal distributed data. Next, Section \ref{powerspike} gives the theoretical analysis of their asymptotic power  under the alternative of a spiked population model. Section \ref{ge} generalizes our test procedures to  populations with an unknown mean.  Technical proofs and
calculations are relegated to
Section~\ref{sec:proofs}. The last Section contains some concluding remarks.

\section{Large-dimensional corrections}
\label{sec:corrected}

From now on, we assume that the observations $Y_1,\ldots,Y_n$ have
the representation $Y_j=\Sigma_p^{1/2} X_j$ where the $p\times n $
table $\{X_1,\ldots,X_n\}=\{x_{ij}\}_{1 \le i\le p,1\le j\le n}$ are
made with an array of i.i.d. standardized random variables (mean 0 and
variance 1).  This setting is motivated by the random matrix theory
and it is generic enough for a precise analysis of the
sphericity test.
Furthermore, under the null hypothesis $H_0: \Sigma_p=\sigma^2I_p$
($\sigma^2$ is unspecified), we notice that both LRT and John's
test are
independent from the scale parameter $\sigma^2$ under the null.
Therefore,  we can assume w.l.o.g. $\sigma^2=1$ when dealing with the
null distributions of these test statistics.  This will be assumed in
all the sections.

Throughout the paper  we will use an indicator $\kappa$ set  to 2
when $\{x_{ij}\}$ are real and to 1 when they are complex as defined in \cite{r3}.
Also, we define  the kurtosis coefficient $\beta=E|x_{ij}|^{4}-1-\kappa$ for both cases and
note that for normal variables,
$\beta=0$ (recall that for a standard complex-valued normal random variable, its real and imaginary parts are two iid. $N(0, \frac12)$ real random variables).

\subsection{The corrected likelihood ratio test (CLRT)}
\label{ssec: CLRT}
For the correction of LRT, let $\mathcal L_n=-2n^{-1}\log L_n~$ be the test statistic for $n\ge 1$.
Our first main result is the following.

\begin{theorem}\label{clt}
  Assume $\{x_{ij}\}$ are iid, satisfying $Ex_{ij}=0, E|x_{ij}|^{2}=1,
  E|x_{ij}|^{4}< \infty$.
  Then under
  $H_{0}$ and  when
  $\frac{p}{n}=y_{n}\rightarrow y\in (0,1)$,
  \begin{eqnarray}\label{a1}
  &&\mathcal L_n + (p-n)\cdot\log(1-\frac{p}{n})-p\nonumber\\
  &&\Longrightarrow
  N\{-\frac{\kappa-1}{2}\log(1-y) + \frac12\beta y,
  -\kappa\log(1-y)-\kappa y \}~.
 \end{eqnarray}
\end{theorem}

The test based on this asymptotic
normal distribution
will be hereafter referred as the
{\em  corrected likelihood-ratio test} ({\em{CLRT}}).
One may observe that the limiting distribution of the test
crucially depends on the limiting dimension-to-sample ratio $y$
through the factor $-\log(1-y)$.
In particular, the asymptotic variance will  blow up quickly
when $y$ approaches  1, so
it is expected that the power
will seriously break down.
Monte-Carlo  experiments in Section~\ref{sec:simul} will provide
more details on this behavior.

The proof of  Theorem~\ref{clt} is based on  the following lemma.
In all the following, $F^y$ denotes the Mar\v{c}enko-Pastur
distribution of index $y$ ($>0$)  which is introduced and discussed
in the Appendix. \reu{And $F^{y}(f(x))=\int f(x) F^y(dx)$ denotes the integral of function $f(x)$ with respect to $F^y$.}

\begin{lemma}\label{lem:joint}
Let $\{\ell_i\}_{1\le i\le p} $ be
the eigenvalues of the sample covariance matrix $S_{n}=n^{-1}\sum_i
Y_iY_i^*$.
Then under $H_{0}$ and the conditions of Theorem~\ref{clt},
we have
\[
\left(
\begin{array}{cc}
  \sum_{i=1}^{p}\log \ell_i -pF^{y_{n}}(\log x)
  \\
  \sum_{i=1}^{p}\ell_i -pF^{y_{n}}(x)
  \\
\end{array}
\right) \Longrightarrow
N(\mu_1,V_1),
\]
with
\[
\mu_1=
\begin{pmatrix}
  \frac{\kappa-1}{2}\log (1-y) - \frac12 \beta y
  \\
  0
\end{pmatrix} ~,
\]
and
\[
V_1 =
\begin{pmatrix}
  -\kappa\log(1-y)+ \beta y
  &
  (\beta+\kappa)y
  \\
  (\beta+\kappa)y
  &   (\beta+\kappa)y
\end{pmatrix}  ~.
\]
\end{lemma}

\noindent The proof of this lemma is postponed to Section~\ref{sec:proofs}.

\medskip

\begin{proof} (of Theorem~\ref{clt}). \quad
Let $A_{n}\equiv \sum_{i=1}^{p}\log \ell_i -pF^{y_{n}}(\log x)$
and $B_{n}\equiv \sum_{i=1}^{p}\ell_i -pF^{y_{n}}(x)$. By
Lemma~\ref{lem:joint},
\[
\left(
\begin{array}{cc}
  A_{n}
  \\
  B_{n}
  \\
\end{array}
\right) \Longrightarrow
N(\mu_1,V_1).
\]
Consequently,
$-A_n+B_n$ is  asymptotically normal  with mean
$-\frac{\kappa-1}{2}\log(1-y)+\frac{1}{2}\beta y$
and variance
\[
  V_1(1,1)+V_1(2,2)-2V_1(1,2)=-\kappa\log(1-y)-\kappa y.
\]
Besides,
\begin{eqnarray*}
\mathcal L_n&=&-\Sigma_{i=1}^{p}\log \ell_i +p\log(\frac{1}{p}\Sigma_{i=1}^{p}\ell_i )\\
&=&-(A_{n}+pF^{y_{n}}(\log x))+p\log (\frac{1}{p}(B_{n}+p))\\
&=&-A_{n}-pF^{y_{n}}(\log x)+p\log (1+\frac{B_{n}}{p}).\\
\end{eqnarray*}
Since $ B_{n} \Rightarrow N(0,y(\beta+\kappa))$,
$B_{n}=O_{p}(1)$ and
$ \log(1+{B_{n}}/{p})={B_{n}}/{p}+O_{p}({1}/{p^{2}}).$ Therefore,
\[\mathcal L_n  =-A_n-pF^{y_{n}}(\log x)+B_n+O_{p}(\frac{1}{p})~.
\]
The conclusion follows with the following well-known integrals
w.r.t. the Mar\v{c}enko-Pastur distribution $F^y$ ($y<1$)  \reu{in \cite{r2}},
\[
F^{y}(\log x)=\frac{y-1}{y}\log(1-y)-1~,  \quad
F^{y}(x)=1.
\]
The proof of Theorem \ref{clt} is complete.
\end{proof}

\subsection{The corrected John's test (CJ)}
\label{ssec:CN}

Earlier than the asymptotic expansion \eqref{exp-Nagao} given in
\cite{r19},
\cite{r10} proved
that when the observations are normal,
the sphericity test based on
$T_2$  is a locally most powerful invariant test.
It is also established in
\cite{r11} that
under these conditions, the limiting distribution of $T_2$
under $H_0$ is $\chi^2_f$ with degree of freedom
$f=\frac12p(p+1)-1 $, or equivalently,
\begin{eqnarray*}
  n U-p\Longrightarrow \frac{2}{p}\chi^2_f
  -p~,
\end{eqnarray*}
where
for convenience, we have let
$U=2(np)^{-1}T_2$.
Clearly, this limit has been  established for $n\to\infty$ and a fixed dimension
$p$. However, if we now let $p\to\infty$ in the right-hand side
of the  above result,
it is not hard to see that
$\frac{2}{p}\chi^2_f  -p$ will tend to the normal distribution
$N(1,4)$.
It  then seems ``natural'' to conjecture that when both
$p$ and $n$ grow to infinity in some  ``proper'' way, it may happen that
\begin{eqnarray}\label{LW}
  n U-p\Longrightarrow N(1,4)~.
\end{eqnarray}
This is indeed the main result of \cite{r15}
where this asymptotic distribution was established
assuming that data are normal-distributed and
$p$ and $n$ grow to infinity in a proportional way (i.e.
$p/n\to y>0$).

In this section, we provide a more general result using our own
approach. In particular, the distribution of the observation is
arbitrary provided a finite fourth moment exists.

\begin{theorem}\label{cnclt}
Assume $\{x_{ij}\}$ are iid, satisfying $Ex_{ij}=0, E|x_{ij}|^{2}=1,
E|x_{ij}|^{4}< \infty$, and let $U=2(np)^{-1}T_2$ be the test statistic. Then under
$H_{0}$ and when $p\rightarrow \infty, n\rightarrow \infty,
\frac{p}{n}=y_{n}\rightarrow y\in (0,\infty)$,
\begin{eqnarray}\label{a2}
n U-p\Longrightarrow N(\kappa+\beta-1,2\kappa)~.
\end{eqnarray}
\end{theorem}

The test based on the asymptotic normal distribution
given in equation~\eqref{a2} will be hereafter referred as
the {\em corrected John's test (CJ)}.

A striking fact in this theorem is that as in the normal case,
the limiting distribution of CJ
is {\em independent} of  the dimension-to-sample ratio $y=\lim p/n$.
In particular, the limiting distribution derived under classical scheme ($p$ fixed, $n\rightarrow\infty$), e.g. the distribution $\frac2p \chi_{f}^2-p$  in the normal case, when used for large $p$, stays very close to this limiting distribution derived for large-dimensional scheme ($p\rightarrow\infty, n\rightarrow\infty, p/n\rightarrow y \in (0,\infty)$).
In this sense, Theorem~\ref{cnclt}
gives a theoretic explanation to the
widely observed robustness of John's
test against the dimension inflation.  Moreover,
CJ is also valid for the $p$ larger (or much larger)
than $n$ case in contrast to the CLRT where this ratio
should be kept smaller than 1 to avoid null eigenvalues.

It is also worth noticing that for real normal data,
we have  $\kappa=2$ and
$\beta=0$ so that the theorem above reduces to $n U-p\Rightarrow
N(1,4)$. This  is exactly the result discussed in
\cite{r15}. Besides, if the data has
a non-normal
distribution but has  the same first four moments as the
normal distribution, we have again
$n U-p\Rightarrow  N(1,4)$, which turns out to have a universality property.

The proof of Theorem~\ref{cnclt}
is based on the  following lemma.

\begin{lemma}\label{cnlemma:joint}
  Let $\{\ell_i\}_{1\le i\le p} $ be
  the eigenvalues of the sample covariance matrix $S_{n}=n^{-1}\sum_i
  Y_iY_i^*$.
  Then under $H_{0}$ and the conditions of Theorem~\ref{cnclt},
  we have
\[
\left(
\begin{array}{cc}
  \sum_{i=1}^{p}\ell_i ^2-p(1+y_n)
  \\
  \sum_{i=1}^{p}\ell_i -p
  \\
\end{array}
\right) \Longrightarrow
N(\mu_2,V_2),
\]
with
\[
\mu_2=
\begin{pmatrix}
  (\kappa-1+\beta)y
  \\
  0
\end{pmatrix} ~,
\]
and
\[
V_2 =
\begin{pmatrix}
  2\kappa y^2+4(\kappa+\beta)(y+2y^2+y^3)
  &
  2(\kappa+\beta)(y+y^2)
  \\
  2(\kappa+\beta)(y+y^2)
  &   (\kappa+\beta)y
\end{pmatrix}  ~.
\]
\end{lemma}

\noindent The proof of this lemma is postponed to Section~\ref{sec:proofs}.
\begin{proof} (of Theorem~\ref{cnclt}).\quad
The result of Lemma~\ref{cnlemma:joint} can be rewritten as:\\
\[n
\left(
\begin{array}{cc}
  p^{-1}\sum_{i=1}^p \ell_i ^2-1-\frac{p}{n}-\frac{(\kappa+\beta-1)y}{p} \\
  p^{-1}\sum_{i=1}^p \ell_i -1
\end{array}
\right)
\Longrightarrow
N\Bigg(
\left(
\begin{array}{cc}
  0  \\  0
\end{array}
\right), \
\frac{1}{y^2}\cdot V_2
\Bigg).
\]
Define the function $f(x,y)=\frac{x}{y^2}-1$, then $U=f(p^{-1}\Sigma_{i=1}^{p}\ell_i ^2,~p^{-1}\Sigma_{i=1}^{p}\ell_i )$.

\noindent We have
\begin{eqnarray*}
&&\frac{\partial f}{\partial x}(1+\frac{p}{n}+\frac{(\kappa+\beta-1)y}{p}, 1)=1~,\\
&&\frac{\partial f}{\partial y}(1+\frac{p}{n}+\frac{(\kappa+\beta-1)y}{p}, 1)=-2(1+\frac{p}{n}+\frac{(\kappa+\beta-1)y}{p})~,\\
&&f(1+\frac{p}{n}+\frac{(\kappa+\beta-1)y}{p}, 1)=\frac{p}{n}+\frac{(\kappa+\beta-1)y}{p}~.\\
\end{eqnarray*}
By the delta method,
\[n\Big(U-f(1+\frac{p}{n}+\frac{(\kappa+\beta-1)y}{p}, 1)\Big)\Longrightarrow
N(0,\lim C),\]
where\\
$C=
$$
 \left(
\begin{array}{cc}
\frac{\partial f}{\partial x}(1+\frac{p}{n}+\frac{(\kappa+\beta-1)y}{p}, 1)
 \\
\frac{\partial f}{\partial y}(1+\frac{p}{n}+\frac{(\kappa+\beta-1)y}{p}, 1)
\\
 \end{array}
\right)^T \
$$\cdot
\Big(\frac{1}{y^2} V_2 \Big)\cdot
$$
 \left(
\begin{array}{cc}
\frac{\partial f}{\partial x}(1+\frac{p}{n}+\frac{(\kappa+\beta-1)y}{p}, 1)
 \\
\frac{\partial f}{\partial y}(1+\frac{p}{n}+\frac{(\kappa+\beta-1)y}{p}, 1)
\\
 \end{array}
\right) \
$$
$\\

$\longrightarrow 2\kappa~.$\\

\noindent Therefore,
\begin{eqnarray*}
n(U- \frac{p}{n}-\frac{(\kappa+\beta-1)y}{p})\Longrightarrow N(0, 2\kappa)~,\\
\end{eqnarray*}
that is,
\begin{eqnarray*}
 n U-p\Longrightarrow N(\kappa+\beta-1,2\kappa)~.\\
 \end{eqnarray*}
The proof of Theorem \ref{cnclt} is complete.
\end{proof}

\begin{remark}
Note that in \reu{Theorems} \ref{clt} and \ref{cnclt}  appears the parameter $\beta$, which is in practice unknown  with real data. So
we may estimate the parameter $\beta$
using the fourth-order  sample moment:
\[ \widehat\beta= \frac{1}{np} \sum_{i=1}^p\sum_{j=1}^n|Y_j(i)|^4
 - \kappa -1~.
\]
According to the law of large numbers, $\hat{\beta}=\beta+o_p(1)$, so substituting $\hat{\beta}$ for $\beta$ in \reu{Theorems} \ref{clt} and \ref{cnclt} does not modify the limiting distribution.
\end{remark}

\section{Monte Carlo study}\label{sec:simul}
Monte Carlo simulations are conducted to find  empirical sizes and
powers of CLRT and CJ.
In particular, here
we want to examine the following  questions:
how robust are the tests against non-normal distributed data and
what is  the range of the dimension to sample ratio $p/n$
where the tests are applicable.

For comparison, we show both the performance of the LW  test
using the asymptotic $N(1,4)$ distribution in \eqref{LW}
(Notice however this is the CJ test under normal distribution) and the Chen's test (denoted as {\em{C}} for short) using the asymptotic $N(0,4)$ distribution derived in \cite{r9}.
The nominal
test level is set to be $\alpha=0.05$, and for each pair of $(p,n)$, we run 10000 independent replications.

We consider two scenarios with respect to the random vectors $Y_i$~:
\begin{enumerate}
\item[(a)]
  $Y_i$ is $p$-dimensional real random vector from the multivariate normal
  population ${\mathcal{N}}(0, I_p)$. In this case, $\kappa=2$ and $\beta=0$.
\item[(b)]
  $Y_i$ consists of iid real random variables with distribution $Gamma(4,2)-2$ so that $y_{ij}$ satisfies $E y_{ij}=0, E y_{ij}^4=4.5$.  In this case, $\kappa=2$ and $\beta=1.5$.
\end{enumerate}

\medskip

\begin{table}[htbp]
\centering
\caption{\label{SIZE}Empirical sizes of LW, CJ, CLRT and C test
  at $5\%$ significance level based on 10000 independent applications
  with real $N(0,1)$ random variables
  and with real Gamma(4,2)-2 random variables.}
\begin{tabular}{ccccccccccc}
\hline
\multirow{2}{*}{$(p, n)$}  & & \multicolumn{3}{c}{$N(0, 1)$} &  & & \multicolumn{4}{c}{Gamma(4,2)-2} \\
\cline{3-5}\cline{8-11}
& &  LW/CJ & CLRT & C  & & & LW & CLRT & CJ & C\\
\hline
(4,64) & & 0.0498 &	0.0553 & 0.0523 & & & 0.1396 & 0.074 & 0.0698 & 0.0717\\
(8,64) & & 0.0545 & 0.061 & 0.0572 & & & 0.1757 & 0.0721 & 0.0804 & 0.078\\
(16,64) & &	0.0539 & 0.0547 & 0.0577 & & & 0.1854 & 0.0614 & 0.078 & 0.0756\\
(32,64) & &	0.0558 & 0.0531 & 0.0612 & & & 0.1943 & 0.0564 & 0.0703 & 0.0682\\
(48,64) & & 0.0551 & 0.0522 & 0.0602 & & & 0.1956 & 0.0568 & 0.0685 & 0.0652\\
(56,64) & & 0.0547 & 0.0505 & 0.0596 & & & 0.1942 & 0.0549 & 0.0615 & 0.0603\\
(60,64) & & 0.0523 & 0.0587 & 0.0585 & & & 0.194 & 0.0582 & 0.0615 & 0.0603\\
\hline
(8,128) & & 0.0539 & 0.0546 & 0.0569 & & & 0.1732 & 0.0701 & 0.075 & 0.0754\\
(16,128) & & 0.0523	& 0.0534 & 0.0548 & & & 0.1859 & 0.0673 & 0.0724 & 0.0694\\
(32,128) & & 0.051	& 0.0545 & 0.0523 & & & 0.1951 & 0.0615 & 0.0695 & 0.0693\\
(64,128) & & 0.0538	& 0.0528 & 0.0552 & & & 0.1867 & 0.0485 & 0.0603 & 0.0597\\
(96,128) & & 0.055	& 0.0568 & 0.0581 & & & 0.1892 & 0.0539 & 0.0577 & 0.0579\\
(112,128) & & 0.0543 & 0.0522 & 0.0591 & & & 0.1875 & 0.0534 & 0.0591 & 0.0593\\
(120,128) & & 0.0545 & 0.0541 & 0.0561 & & & 0.1849 & 0.051 & 0.0598 & 0.0596\\
\hline
(16,256) & & 0.0544	& 0.055 & 0.0574 & & & 0.1898 & 0.0694 & 0.0719 & 0.0716\\
(32,256) & & 0.0534	& 0.0515 & 0.0553 & & & 0.1865 & 0.0574 & 0.0634 & 0.0614\\
(64,256) & & 0.0519	& 0.0537 & 0.0522 & & & 0.1869 & 0.0534 & 0.0598 & 0.0608\\
(128,256) & & 0.0507 & 0.0505 & 0.0498 & & & 0.1858 & 0.051 & 0.0555 & 0.0552\\
(192,256) & & 0.0507 & 0.054 & 0.0518 & & & 0.1862 & 0.0464 & 0.052 & 0.0535\\
(224,256) & & 0.0503 & 0.0541 & 0.0516 & & & 0.1837 & 0.0469 & 0.0541 & 0.0538\\
(240,256) & & 0.0494 & 0.053 & 0.0521 & & & 0.1831 & 0.049 & 0.0533 & 0.0559\\
\hline
(32,512) & & 0.0542	& 0.0543 & 0.0554 & & & 0.1884 & 0.0571 & 0.0606 & 0.059\\
(64,512) & & 0.0512	& 0.0497 & 0.0513 & & & 0.1816 & 0.0567 & 0.0579 & 0.0557\\
(128,512) & & 0.0519& 0.0567 & 0.0533 & & & 0.1832 & 0.0491 & 0.0507 & 0.0504\\
(256,512) & & 0.0491 & 0.0503 & 0.0501 & & & 0.1801 & 0.0504 & 0.0495 & 0.0492\\
(384,512) & & 0.0487 & 0.0505 & 0.0499 & & & 0.1826 & 0.051 & 0.0502 & 0.0507\\
(448,512) & & 0.0496 & 0.0495 & 0.0503 & & & 0.1881 & 0.0526 & 0.0482 & 0.0485\\
(480,512) & & 0.0488 & 0.0511 & 0.0505 & & & 0.1801 & 0.0523 & 0.053 & 0.0516\\
\hline
\end{tabular}
\end{table}

Table \ref{SIZE} reports the sizes of the four tests in these two
scenarios for different values of $(p, n)$. We see that when $\{y_{ij}\}$
are normal, LW (=CJ), CLRT and C all
have similar empirical sizes tending to the nominal level 0.05
as either  $p$ or $n$ increases. But when $\{y_{ij}\}$ are Gamma-distributed,
the sizes of LW are higher than 0.1 no matter how large
the values of $p$ and $n$ are while the sizes of
 CLRT and CJ all converge to the nominal level 0.05 as either
$p$ or  $n$ gets larger. This empirically confirms that
normal assumptions
are needed for the result of  \cite{r15}
while our corrected statistics CLRT and CJ (also the C test) have no such distributional
restriction.

As for  empirical powers, we consider  two alternatives (here, the limiting spectral distributions of $\Sigma_p$ under these two alternatives differs from that under $H_0$):
\begin{enumerate}
\item[(1)]
  $\Sigma_p$ is diagonal with half of its diagonal elements 0.5 and half 1. We denote its power
    by  Power 1~;
\item[(2)]
  $\Sigma_p$ is diagonal with 1/4 of the elements equal  0.5 and 3/4 equal 1.
   We denote its power by  Power 2~.
\end{enumerate}

Table \ref{POWER} reports the powers of LW(=CJ),
CLRT and C when $\{y_{ij}\}$ are distributed as $N(0,1)$, and of CJ, CLRT and C when
$\{y_{ij}\}$ are  distributed as Gamma(4,2)-2,  for the situation when $n$
equals $64$ or $128$, with varying values of $p$  and  under the above mentioned
two alternatives.
For $n=256$ and $p$ varying
from 16 to 240, all the tests have powers around 1 under both
alternatives so that these values are omitted. And in order to find the trend of these
powers, we also present the results when $n=128$ in Figure \ref{normal128} and Figure \ref{gamma128}.

\begin{table}[htbp]
\centering
\caption{\label{POWER}Empirical powers of LW, CJ, CLRT and C test
  at $5\%$ significance level based on 10000 independent applications
  with real $N(0,1)$ random variables
  and with real Gamma(4,2)-2 random variables under two  alternatives Power 1 and 2 (see the text for details).}
\begin{tabular}{cccccccccc}
\multicolumn{10}{c}{$N(0, 1)$}\\
\hline
\multirow{2}{*}{$(p, n)$} & & \multicolumn{3}{c}{Power 1} & & & \multicolumn{3}{c}{Power 2}\\
\cline{3-5}\cline{8-10}
& & LW/CJ &  CLRT & C & & & LW/CJ & CLRT & C\\
\hline
(4,64) & & 0.7754 & 0.7919 & 0.772 & & & 0.4694 & 0.6052 & 0.4716\\
(8,64) & & 0.8662 & 0.8729 & 0.8582 & & & 0.5313 & 0.6756 & 0.5308\\
(16,64)	& & 0.912 & 0.9075 & 0.9029 & & & 0.5732 & 0.6889 & 0.5671\\
(32,64)	& & 0.9384 & 0.8791 & 0.931 & & & 0.5868 & 0.6238 & 0.5775\\
(48,64)	& & 0.9471 & 0.7767 & 0.9389 & & & 0.6035 & 0.5036 & 0.5982\\
(56,64)	& & 0.949 & 0.6663 & 0.9411 & & & 0.6025 & 0.4055 & 0.5982\\
(60,64)	& & 0.9501 & 0.5575 & 0.941 & & & 0.6048 & 0.3328 & 0.5989\\
\hline
(8,128)	& & 0.9984 & 0.9989 & 0.9986 & & & 0.9424 & 0.9776 & 0.9391\\
(16,128) & & 0.9998 & 1 & 0.9998 & & & 0.9698 & 0.9926 & 0.9676\\
(32,128) & & 1 & 1 & 1 & & & 0.9781 & 0.9956 & 0.9747\\
(64,128) & & 1 & 1 & 1 & & & 0.9823 & 0.9897 & 0.9788\\
(96,128) & & 1 & 0.9996 & 1 & & & 0.9824 & 0.9532 & 0.9804\\
(112,128) & & 1 & 0.9943 & 1 & & & 0.9841 & 0.881 & 0.9808\\
(120,128) & & 1 & 0.9746 & 1 & & & 0.9844 & 0.7953 & 0.9817\\
\hline
\\
\\
\multicolumn{10}{c}{$Gamma(4,2)-2$}\\
\hline
\multirow{2}{*}{$(p, n)$}& & \multicolumn{3}{c}{Power 1} & & & \multicolumn{3}{c}{Power 2}\\
\cline{3-5}\cline{8-10}
& & CJ &  CLRT & C & & & CJ & CLRT & C\\
\hline
(4,64) & & 0.6517 & 0.6826 & 0.6628 & & & 0.3998 & 0.5188 & 0.4204\\
(8,64) & & 0.7693 & 0.7916 & 0.781 & & & 0.4757 & 0.5927 & 0.4889\\
(16,64) & & 0.8464 & 0.8439 & 0.846 & & & 0.5327 & 0.633 & 0.5318\\
(32,64) & & 0.9041 & 0.848 & 0.9032 & & & 0.5805 & 0.5966 & 0.5667\\
(48,64) & & 0.9245 & 0.7606 & 0.924 & & & 0.5817 & 0.4914 & 0.5804\\
(56,64) & & 0.9267 & 0.6516 & 0.9247 & & & 0.5882 & 0.4078 & 0.583\\
(60,64) & & 0.9288 & 0.5547 & 0.9257 & & & 0.5919 & 0.3372 & 0.5848\\
\hline
(8,128) & & 0.9859 & 0.9875	& 0.9873 & & & 0.8704 & 0.9294 & 0.8748\\
(16,128) & & 0.999 & 0.999 & 0.9987 & & & 0.9276 & 0.9699 & 0.9311\\
(32,128) & & 0.9999 & 1 & 0.9999 & & & 0.9582 & 0.9873 & 0.9587\\
(64,128) & & 1 & 0.9998 & 1 & & & 0.9729 & 0.984 & 0.9727\\
(96,128) & & 1 & 0.999 & 1 & & & 0.9771 & 0.9482 & 0.9763\\
(112,128) & & 1 & 0.9924 & 1 & & & 0.9781 & 0.8747 & 0.9763\\
(120,128) & & 1 & 0.9728 & 1 & & & 0.9786 & 0.7864 & 0.977\\
\hline
\end{tabular}
\end{table}

The
behavior of Power 1 and Power 2 in each figure related to the three
statistics are similar, except that Power 1 is much higher compared with
Power 2 for a given dimension design $(p,n)$ and any given test for the reason that the
first alternative differs more from the null than the second one. The powers of LW
(in the normal case), CJ (in the Gamma case) and C are all monotonically
increasing in  $p$ for a  fixed value of $n$. But for CLRT,
when $n$ is fixed, the powers first increase in $p$  and then
become decreasing  when
 $p$ is getting close to $n$. This can be explained  by the fact
that when
$p$ is close to $n$, some of the eigenvalues of $S_n$ are getting
close to zero, causing the CLRT nearly degenerate and losing power.

\begin{figure}[h!]
\centering
\begin{minipage}[c]{0.5\textwidth}
\centering
\includegraphics[width=6.5cm]{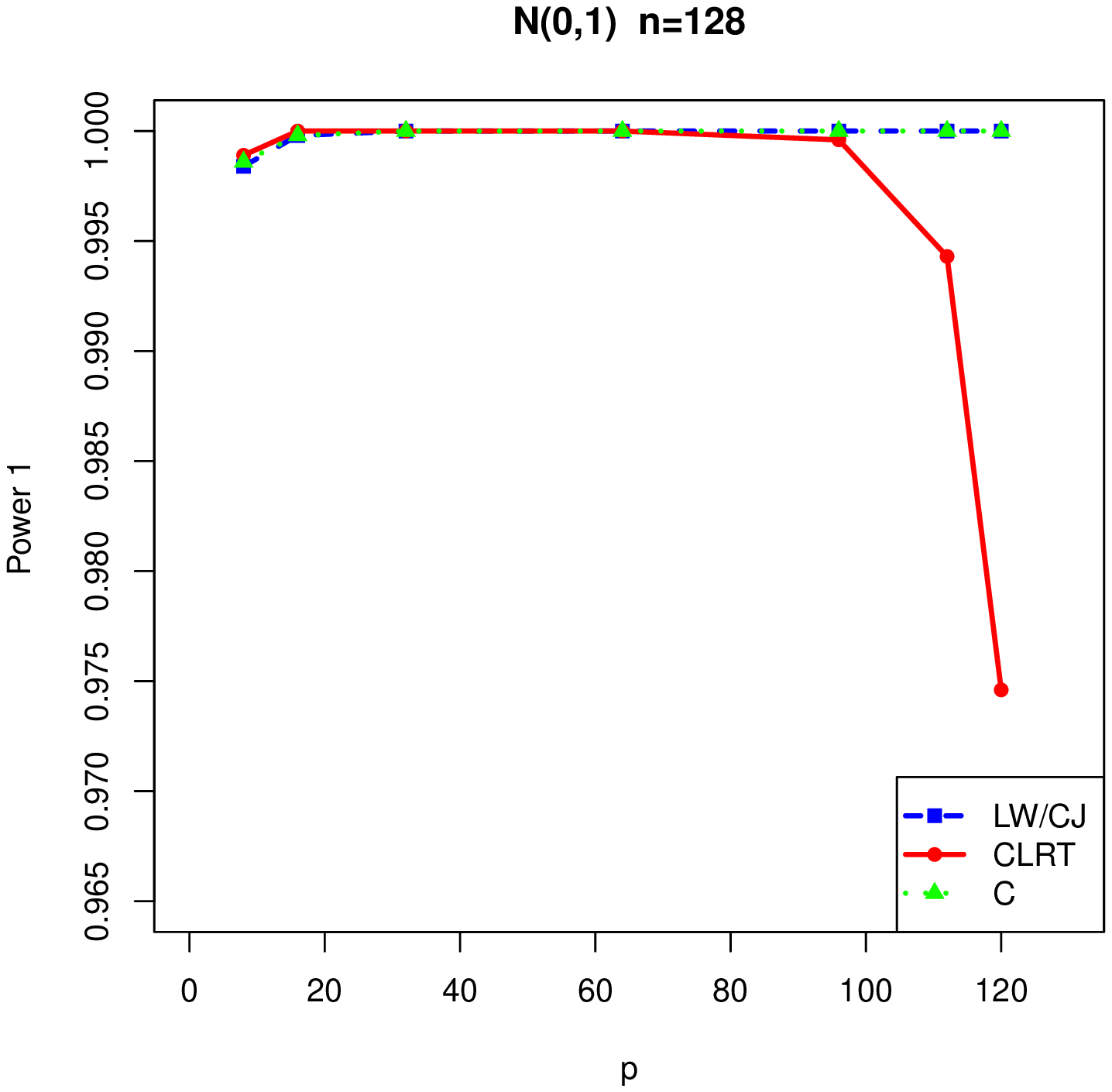}
\end{minipage}%
\begin{minipage}[c]{0.5\textwidth}
\centering
\includegraphics[width=6.5cm]{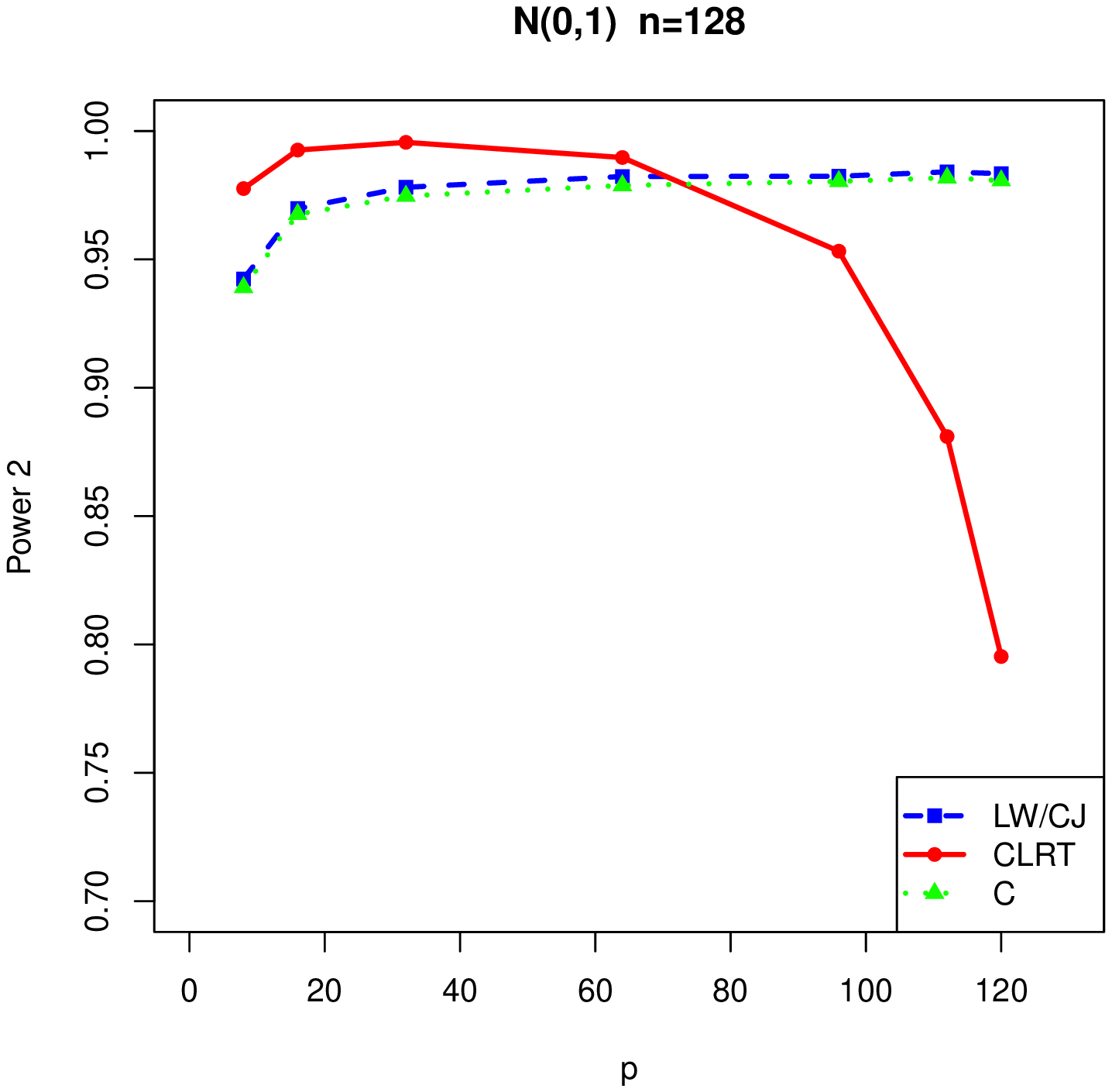}
\end{minipage}
\caption{\small Empirical powers of LW/CJ, CLRT and C test at $5\%$ significance
  level based on 10000 independent applications with real
  N(0,1) random variables for fixed $n=128$  under two  alternatives
  Power 1 and  2 (see the text for details). }\label{normal128}
\end{figure}

\begin{figure}[h!]
\centering
\begin{minipage}[c]{0.5\textwidth}
\centering
\includegraphics[width=6.5cm]{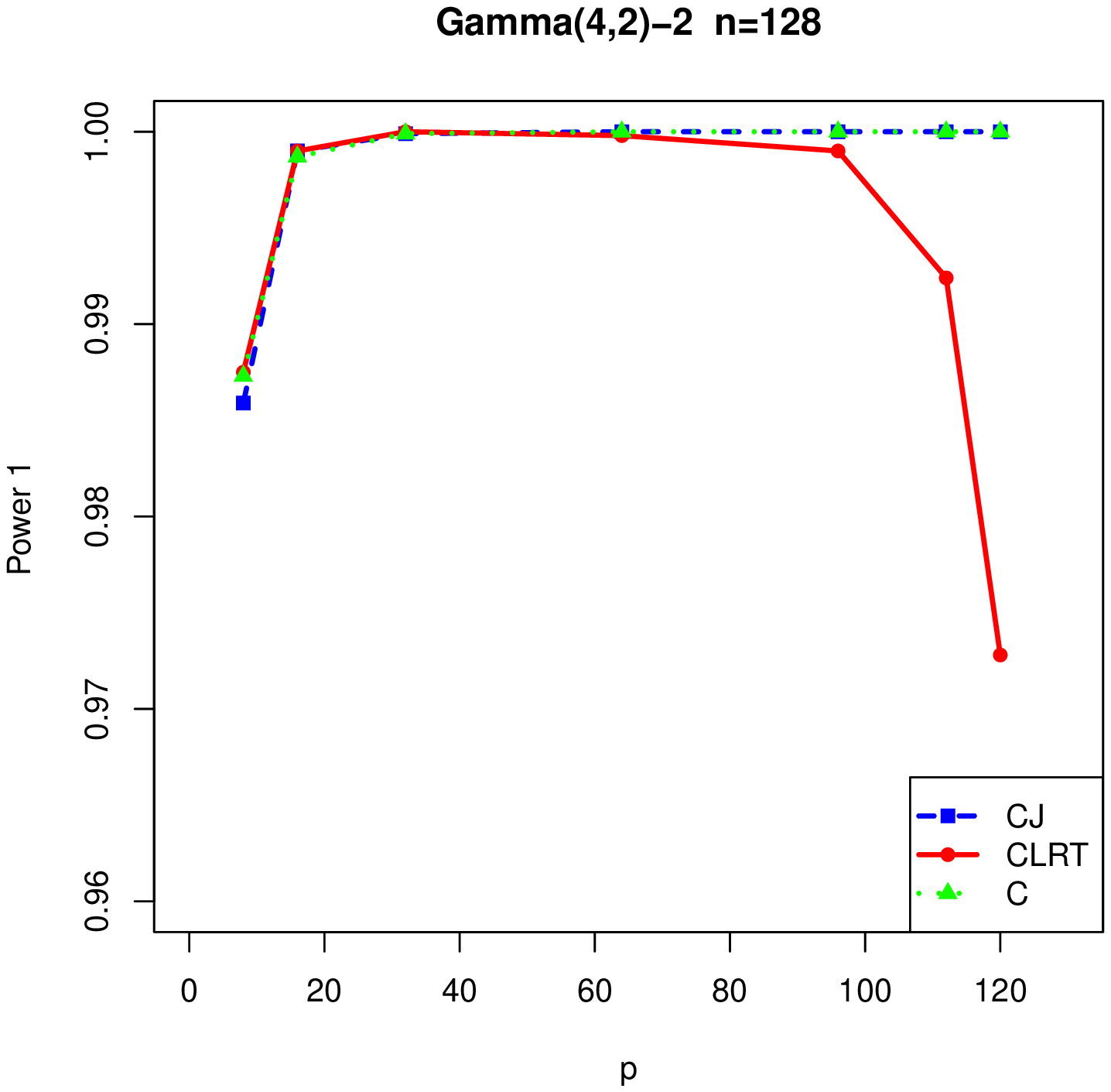}
\end{minipage}%
\begin{minipage}[c]{0.5\textwidth}
\centering
\includegraphics[width=6.5cm]{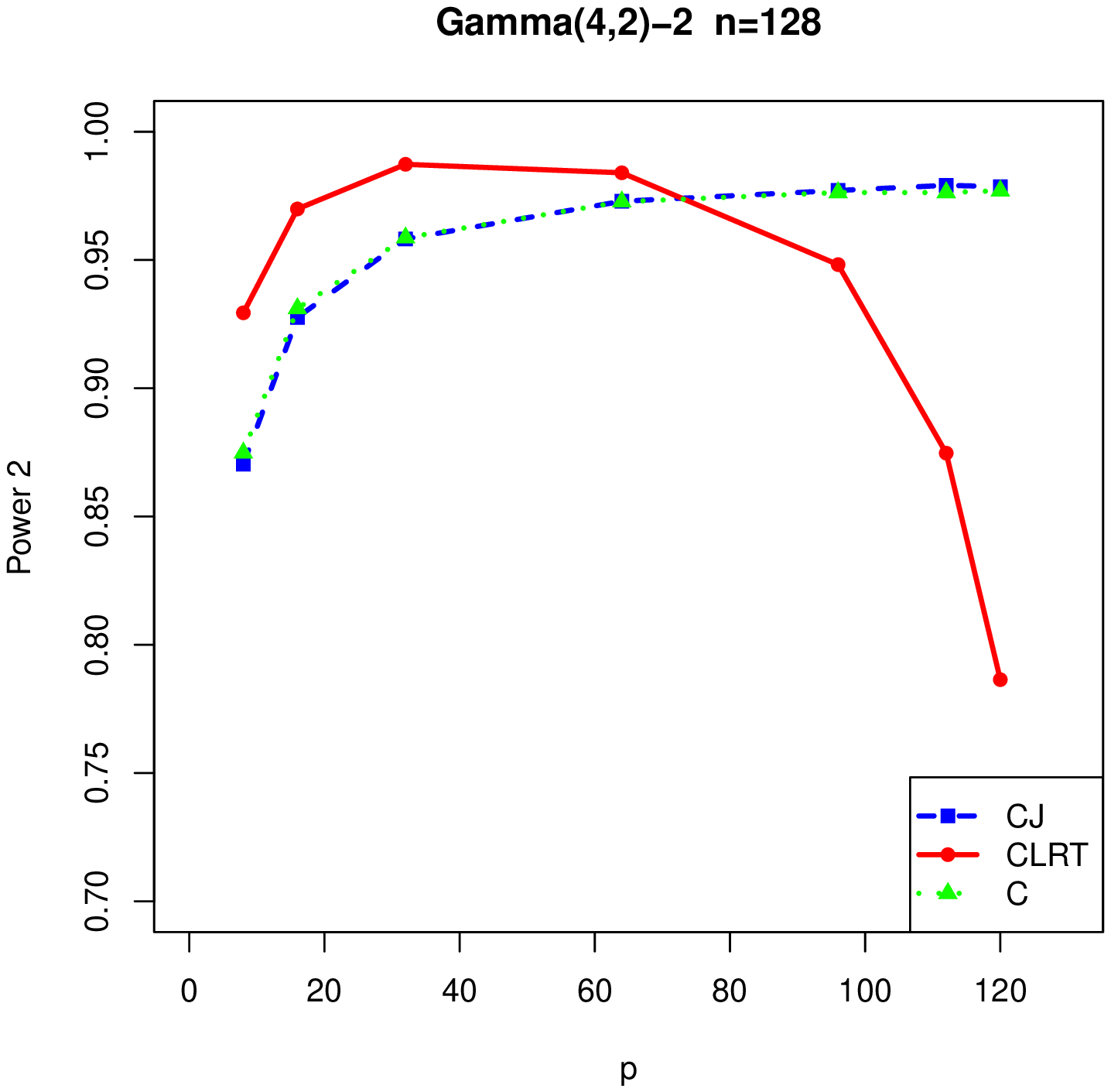}
\end{minipage}
\caption{\small Empirical powers of CJ, CLRT and C test at $5\%$ significance
  level based on 10000 independent applications with real Gamma(4,2)-2
  random variables for fixed $n=128$  under two  alternatives Power 1 and  2 (see the text for details). }\label{gamma128}
\end{figure}

Besides, we find that in the normal case
the trend of C's power is very much alike of those
of LW while in the Gamma case it is similar with those of CJ under both alternatives. And in most of the cases (especially in large $p$ case), the power of C test is slightly lower
 than LW (in the normal case) and CJ (in the Gamma case).

Lastly, we examine  the performance of CJ and C when $p$ is larger than $n$.
Empirical sizes and powers  are presented in Table
\ref{cnsizepower}.
 We choose the variables to be distributed as
Gamma(4,2)-2 since  CJ reduces to  LW in the normal case, and
\cite{r15} has already reported  the performance of LW when
$p$ is larger than $n$. From the table, we see that when $p$ is larger than $n$,
the size of CJ is still correct and it is always around the nominal level 0.05 as the dimension $p$ increases
and the same phenomenon exists for C test.

\begin{table}[h]
\centering
\caption{\label{cnsizepower}Empirical sizes and powers (Power 1 and 2) of CJ test and C test at
  $5\%$ significance level based on 10000 independent applications
  with real Gamma(4,2)-2 random variables when $p\geq n$. }
\begin{tabular}{cccccccccc}
\hline
\multirow{2}{*}{$(p,n)$} & & \multicolumn{3}{c}{CJ} & & & \multicolumn{3}{c}{C}\\
\cline{3-5}\cline{8-10}
 & & Size & Power 1 & Power 2 & & & Size & Power 1 & Power 2\\
\hline
(64,64) & & 0.0624 & 0.9282 & 0.5897 & & & 0.0624 & 0.9257 & 0.5821\\
(320,64) & & 0.0577 & 0.9526 & 0.612 & & & 0.0576 & 0.9472 & 0.6059\\
(640,64) & & 0.0558 & 0.959 & 0.6273 & & & 0.0562 & 0.9541 & 0.6105\\
(960,64) & & 0.0543 & 0.9631 & 0.6259 & & & 0.0551 &  0.955 & 0.6153\\
(1280,64) & & 0.0555 & 0.9607 & 0.6192 & & & 0.0577 & 0.9544 & 0.6067\\
\hline
\end{tabular}
\end{table}

 When we evaluate the power, the same two
alternatives Power 1 and Power 2 as above are considered.
The sample size is fixed to $n=64$ and the ratio $p/n$ varies from
1 to 20.
We see that Power 1 are much  higher than
Power 2
for the
same reason that the first alternative is easier to be distinguished from
$H_0$. Besides, the powers under both alternatives all increase
monotonically for $1\leq \frac{p}{n}\leq 15$. However, when $p/n$ is getting larger, say $p/n=20$, we can observe that its size is a little larger and powers a little drop (compared with $p/n=15$) but overall, it still behaves well, which can be considered as free from the assumption constraint ``$p/n\rightarrow y$''. Besides, the powers of CJ are always slightly higher than those of C in this ``large $p$ small $n$'' setting.

Since the asymptotic distribution for the CLRT
and CJ are both derived under the ``Marcenko-Pasture scheme'' (i.e $p/n\rightarrow y \in (0,\infty)$), if $p/n$ is getting too large ($p \gg n$), it seems that the limiting results provided in this paper will loose accuracy. It is worth noticing that \cite{r7} has extended the LW test to such a scheme ($p \gg n$) for multivariate normal distribution.

Summarizing all these findings from this Monte-Carlo study,
the overall figure is the following: when the ratio
$p/n$ is much lower than 1 (say smaller than 1/2), it is preferable to employ
CLRT (than CJ, LW or C); while this ratio is higher, CJ (or LW for
normal data) becomes more powerful (slightly more powerful than C).

\medskip

\section{Asymptotic powers: under the spiked population alternative}\label{powerspike}
In this section, we give an analysis of the powers of the two corrections: CLRT and CJ. To this end,
 we consider an alternative model
that has attracted lots of attention since its introduction by \cite{r13}, namely, the spiked population model. This model can be described as follows:
the eigenvalues of $\Sigma_p$ are all one except for a few fixed number of them. Thus, we restrict our sphericity testing problem to the following:
\[
H_0:\Sigma_p=I_p~~~vs~~~H_1^{*}: \Sigma_p=diag(\underbrace{a_1,\cdots,a_1}_{n_1},
  \cdots,
  \underbrace{a_k,\cdots,a_k}_{n_k},
  \underbrace{1,\cdots,1}_{p-M}),
\]
where the multiplicity numbers $n_i$'s are fixed and satisfying $n_1+\cdots+n_k=M$. We derive the explicit \reu{expressions} of the power \reu{functions} of CLRT and CJ in this section.

Under $H_1^{*}$, the empirical spectral distribution of $\Sigma_p$ is
\begin{eqnarray}\label{converge}
H_n=\frac{1}{p}\sum_{i=1}^kn_i\delta_{a_i}+\frac{p-M}{p}\delta_1~,
\end{eqnarray}
and it will converge to $\delta_1$, a Dirac mass at 1, which is the same as the limit under the null hypothesis $H_0: \Sigma_p=I_p$.
 From this point of view,
anything related to the limiting spectral distribution remains the same whenever under $H_0$ or $H_1^{*}$. Then recall the CLT for LSS of the sample covariance matrix, \reu{as provided in \cite{r2}}, is of the form:
\[
p\int f(x)d(F^{S_n}(x)-F^{y_n}(x))\Longrightarrow N(\mu, \sigma^2)~,
\]
where the right side of this equation is determined only by the limiting spectral distribution. So we can conclude that the limiting parameters $\mu$ and $\sigma^2$
remain the same under $H_0$ and $H_1^{*}$, only the centering term $p\int f(x)dF^{y_n}(x)$ possibly makes a difference. Since there's a $p$ in front of $\int f(x)dF^{y_n}(x)$, which tends to infinity as assumed, so knowing the convergence $H_n\rightarrow \delta_1$ is not enough and more details about the convergence are needed. In \cite{r28}, we have established an asymptotic expansion for the centering parameter when the population has a spiked structure. We will use these formulas like \reu{equations} \eqref{logx}, \eqref{x} and \eqref{xx} in the following to derive the powers of the CLRT and CJ.

Lemma \ref{lem:joint} remains the same under $H_1^{*}$, except that this time the centering terms become (see formulas (4.11) and (4.12) in \cite{r28}):
\begin{eqnarray}
&&F^{y_n}(\log x)=\frac{1}{p}\sum_{i=1}^k n_i\log a_i-1+(1-\frac{1}{y_n})\log(1-y_n)+O(\frac{1}{p^2})~,\label{logx}\\
&&F^{y_n}(x)=1+\frac{1}{p}\sum_{i=1}^k n_i a_i-\frac{M}{p}+O(\frac{1}{p^2})~.\label{x}
\end{eqnarray}
Repeating the proof of Theorem \ref{clt}, we can get:
\begin{eqnarray}\label{spike1clt}
&&\mathcal{L}_n-p+(p-n)\log(1-\frac{p}{n})+\sum_{i=1}^{k}n_i\big(\log a_i-a_i+1\big)\nonumber\\
&&\Longrightarrow N\Big(-\frac{\kappa-1}{2}\log(1-y)+\frac12\beta y~, ~-\kappa\log(1-y)-\kappa y\Big)
\end{eqnarray}
under $H_1^{*}$. As a result, the power of CLRT for testing $H_0$ against $H_1^{*}$ can be expressed as:
\begin{eqnarray}\label{b1}
\beta_1(\alpha)=1-\Phi\bigg(\Phi^{-1}(1-\alpha)-\frac{\sum_{i=1}^k n_i\big(a_i-\log a_i-1\big)}{\sqrt{-\kappa \log(1-y)-\kappa y}}\bigg)~
\end{eqnarray}
for a pre-given significance level $\alpha$.

\reu{It is worth noticing here that if the alternative has only one simple spike, i.e. $k=1,~n_k=1$, and assuming the real Gaussian variable case, i.e. $\kappa=2$, \eqref{b1} reduces to a result provided in \cite{r20}. However, our formula is valid for a general number of spikes with eventual multiplicities. Besides, these authors use some more sophisticated
tools  of asymptotic contiguity and Le Cam's first and third lemmas, which are totally different from ours.}

In order to calculate the power function of CJ, we restated Lemma \ref{cnlemma:joint} as follows:
\[
\left(
\begin{array}{cc}
  \sum_{i=1}^{p} l_{i}^2-pF^{y_{n}}( x^2)
  \\
  \sum_{i=1}^{p}l_{i}-pF^{y_{n}}(x)
  \\
\end{array}
\right) \Longrightarrow
N(\mu_2,V_2)~,
\]
where
\begin{eqnarray}
&&F^{y_n}(x^2)=\frac{2}{n}\sum_{i=1}^k n_ia_i-\frac{2}{n} M+1+y_n-\frac{M}{p}+\frac1p\sum_{i=1}^k n_i a_i^2+O(\frac{1}{p^2}),\label{xx}\\
&&F^{y_n}(x)=1+\frac{1}{p}\sum_{i=1}^k n_i a_i-\frac{M}{p}+O(\frac{1}{p^2})~.\nonumber
\end{eqnarray}
Using the delta method as in the proof of Theorem \ref{cnclt}, this time, we have
\begin{eqnarray}\label{spike2clt}
nU-p-\frac np \sum_{i=1}^{k}n_i(a_i-1)^2\Longrightarrow N(\kappa+\beta-1, 2\kappa)
\end{eqnarray}
under $H_1^{*}$.

Power function of CJ can be expressed as
\begin{eqnarray}\label{b2}
\beta_2(\alpha)=1-\Phi\bigg(\Phi^{-1}(1-\alpha)-\frac np\cdot\frac{\sum_{i=1}^k n_i(a_i-1)^2}{\sqrt{2\kappa}}\bigg)~
\end{eqnarray}
for a pre-given significance level $\alpha$.

Now consider the \reu{functions} $a_i-\log a_i-1$ and $(a_i-1)^2$ appearing in the \reu{expressions} \eqref{b1} and \eqref{b2}, they will achieve  their minimum value $0$ at $a_i=1$, which is to say, once $a_i$'s going away from 1, the powers $\beta_1(\alpha)$ and $\beta_2(\alpha)$ will both increase. This phenomenon agrees with our intuition, since the more $a_i$'s deviate away from 1, the easier to distinguish $H_0$ from $H_1^{*}$. Therefore, the powers  should naturally grow higher.

\begin{figure}[h]
\centering
\includegraphics[width=7.5cm]{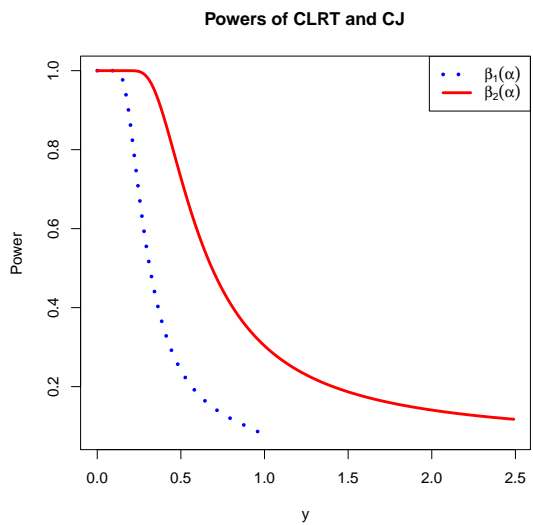}
\caption{\small Theoretical  powers of CLRT ($\beta_1(\alpha)$) and CJ ($\beta_2(\alpha)$) under the spiked population alternative.}\label{beta}
\end{figure}

Then, we consider the power functions $\beta_1(\alpha)$ and $\beta_2(\alpha)$ as  functions of $y$, and see how they are going along with $y's$ changing. we see that in expression \eqref{b1}, $\sqrt{-\log(1-y)-y}$ is increasing when $y \in (0, 1)$, so $\beta_1(\alpha)$ is decreasing as the function of $y$, which attains its maximum value $1$ when $y\rightarrow0_{+}$ and minimum value $\alpha$ when $y\rightarrow1_{-}$. Also, expression \eqref{b2} is obviously a decreasing function of $y$, attaining its maximum value $1$ when $y\rightarrow0_{+}$ and minimum value $\alpha$ when $y\rightarrow +\infty$. We present the trends of $\beta_1(\alpha)$ and $\beta_2(\alpha)$ (corresponding to the power of CLRT and CJ) in Figure \ref{beta} when only one spike $a=2.5$ exists. It is however a little different from the non-spiked case as showed in the simulation in Section \ref{sec:simul} (both Figure \ref{normal128} and Figure \ref{gamma128}), where the power of CLRT first increases then decreases while the power of CJ is always increasing along with the increase of the value of $p$.
These power drops  are due to the fact that when $p$ increases, since only one spike eigenvalue is considered, it becomes harder to distinguish both hypotheses.
Besides, an interesting finding here is that these power functions give a new confirmation of the fact that CLRT behaves quite badly when $y\rightarrow 1_{-}$, while CJ test has a reasonable power for a significant range of $y>1$.

\section{Generalization to the case when the population mean is unknown}\label{ge}
So far, we have assumed the observation $(Y_i)$ are centered. However, this is hardly true in practical situations when $\mu=EY_i$ is usually unknown. Therefore, the sample covariance matrix should be taken as
\[
S_n^{*}=\frac{1}{n-1}\sum_{i=1}^{n}(Y_i-\overline{Y})(Y_i-\overline{Y})^{*}~,
\]
where $\overline{Y}=\frac{1}{n}\sum_{i=1}^{n}Y_i$ is the sample mean.
Because $\overline{Y}\overline{Y}^{*}$ is a rank one matrix, substituting $S_n$ for $S_n^{*}$ when $\mu$ is unknown will not affect the limiting distribution in the CLT for LSS; while it is not the case for the centering parameter, for it has a $p$ in front.

Recently, \cite{r22} shows that if we use $S_n^{*}$ in the CLT for LSS when $\mu$ is unknown, the limiting variance remains the same as we use $S_n$; while the limiting mean has a shift which can be expressed as a complex contour integral. Later, \cite{r30} looks into this shift and finally derives a concise conclusion on the CLT corresponding to $S_n^{*}$:
 the random vector $\big(X_n^{*}(f_1), \cdots, X_n^{*}(f_k)\big)$
converges weakly to a Gaussian vector with the same mean and covariance function as given in  Theorem \ref{t1}, where this time,
$X_n^{*}(f)=p\int f(x)d(F^{S_n^{*}}-F^{y_{n-1}})$. It is here important to pay attention that the only difference is in the centering term, where we use the new ratio $y_{n-1}=\frac{p}{n-1}$ instead of the previous $y_{n}=\frac{p}{n}$, while leaving all the other terms unchanged.

Using this result, we can modify our Theorems \ref{clt} and \ref{cnclt} to get the CLT of CLRT and CJ under $H_0$ when $\mu$ is unknown only by considering the eigenvalues of $S_n^{*}$ and substituting $n-1$ for $n$ in the centering terms. More precisely, now \reu{equations} \eqref{a1} and \eqref{a2} in \reu{Theorems} \ref{clt} and \ref{cnclt} become
\begin{eqnarray*}
  &&\mathcal L_n + (p-n+1)\cdot\log(1-\frac{p}{n-1})-p\nonumber\\
  &&\Longrightarrow
  N\{-\frac{\kappa-1}{2}\log(1-y) + \frac12\beta y,
  -\kappa\log(1-y)-\kappa y \}~
   \end{eqnarray*}
   and
    \begin{eqnarray*}
 n U-\frac{np}{n-1}\Longrightarrow N(\kappa+\beta-1,2\kappa)~.
 \end{eqnarray*}

 The same procedures can be applied to get the CLT of CLRT and CJ under $H_1^{*}$ when $\mu$ is unknown. This time, equations \eqref{spike1clt} and \eqref{spike2clt} become
 \begin{eqnarray*}
&&\mathcal{L}_n-p+(p-n+1)\log(1-\frac{p}{n-1})+\sum_{i=1}^{k}n_i\big(\log a_i-a_i+1\big)\nonumber\\
&&\Longrightarrow N\Big(-\frac{\kappa-1}{2}\log(1-y)+\frac12\beta y~, ~-\kappa\log(1-y)-\kappa y\Big)
\end{eqnarray*}
and
\begin{eqnarray*}
nU-\frac{np}{n-1}-\frac {n}{p} \sum_{i=1}^{k}n_i(a_i-1)^2\Longrightarrow N(\kappa+\beta-1, 2\kappa)~,
\end{eqnarray*}
and therefore the powers of CLRT and CJ under the spiked alternative remain unchanged as expressed in \eqref{b1} and \eqref{b2}.

.

\section{Additional proofs}\label{sec:proofs}
\reu{We recall these two important formulas which appear in the Appendix as \eqref{eq:E} and \eqref{eq:cov} here for the convenience of reading:
  \begin{eqnarray*}
    \E[X_f]& =&   (\kappa-1)  I_1(f) + \beta  I_2(f) ~,\\
    \cov(X_{f},X_{g}) &=& \kappa J_1(f,g) + \beta
    J_2(f,g)~.
  \end{eqnarray*}  }
\subsection{Proof of Lemma~\protect\ref{lem:joint}}
Let for $x>0$, $f(x)=\log x$ and $g(x)=x$.
Define $A_n$ and $B_n$  by the decompositions
\begin{eqnarray*}
  \sum_{i=1}^{p}\log \ell_i & = & p\int f(x)
  d(F_{n}(x)-F^{y_{n}}(x))+pF^{y_{n}}(f) = A_n+pF^{y_{n}}(f)~ , \\
  \sum_{i=1}^{p}\ell_i & = &
  p\int g(x) d(F_{n}(x)-F^{y_{n}}(x))+pF^{y_{n}}(g)= B_n + pF^{y_{n}}(g)  ~.
\end{eqnarray*}
Applying  \reu{Theorem~\ref{t1}}  given in the Appendix to the pair $(f,g)$,
we have
\[
\left(
\begin{array}{cc}
  A_n  \\  B_n
\end{array}
\right)
\Longrightarrow
N\Bigg(
\left(
\begin{array}{cc}
  EX_{f}  \\  EX_{g}
\end{array}
\right), \
\left(
\begin{array}{cc}
  \cov(X_f,X_f)
  &
  \cov(X_f, X_g)
  \\
  \cov(X_g,X_f)
  & \cov(X_g,X_g)
\end{array}
\right) \
\Bigg).
\]
It remains to evaluate the limiting parameters
and this results from the following calculations \reu{where $h$ is denoted as $\sqrt{y}$}:
\begin{eqnarray}
  I_1(f,r)&=& \frac12\log\left( 1-h^2/r^2\right)~,\label{I1f}\\
  I_1(g,r) &=& 0 ~,                               \label{I1g} \\
  I_2(f)&=& -\frac12 h^2~,        \label{I2f} \\
  I_2(g)&=& 0~,                  \label{I2g} \\
  J_1(f,g,r) &=& \frac {h^2}{r^2}~,  \label{J1fg}\\
  J_1(f,f,r) &=& -\frac{1}{r} \log(1-h^2/r) ~,  \label{J1ff}\\
  J_1(g,g,r) &=& \frac {h^2}{r^2}~,  \label{J1gg}\\
  J_2(f,g) &=&  {h^2} ~,  \label{J2fg}\\
  J_2(f,f) &=&  {h^2} ~,  \label{J2ff}\\
  J_2(g,g) &=&  {h^2} ~.  \label{J2gg}
\end{eqnarray}

\noindent We now detail these calculations to complete the proof. They are all
based on the formula given in Proposition~\ref{t2} in the Appendix and
repeated use of the residue theorem.

\medskip\noindent{\em Proof of \eqref{I1f}}:\quad We have
\begin{eqnarray*}
  I_1(f,r)&=& \frac{1}{2\pi i}\oint_{|\xi|=1}f(|1+h\xi|^{2})\Big[\frac{\xi}{\xi^{2}-r^{-2}}-\frac{1}{\xi}\Big]d\xi\\
  &=& \frac{1}{2\pi i}\oint_{|\xi|=1}\log(|1+h\xi|^{2})\Big[\frac{\xi}{\xi^{2}-r^{-2}}-\frac{1}{\xi}\Big]d\xi\\
  &=& \frac{1}{2\pi i}\oint_{|\xi|=1}(\frac{1}{2}\log((1+h\xi)^{2})+\frac{1}{2}\log((1+h\xi^{-1})^{2})\Big[\frac{\xi}{\xi^{2}-r^{-2}}-\frac{1}{\xi}\Big]d\xi\\
  &=& \frac{1}{2\pi i}\Big[\oint_{|\xi|=1}\log(1+h\xi)\frac{\xi}{\xi^{2}-r^{-2}}d\xi-\oint_{|\xi|=1}\log(1+h\xi)\frac{1}{\xi}d\xi\\
    &&+\oint_{|\xi|=1}\log(1+h\xi^{-1})\frac{\xi}{\xi^{2}-r^{-2}}d\xi-\oint_{|\xi|=1}\log(1+h\xi^{-1})\frac{1}{\xi}d\xi\Big]\\
\end{eqnarray*}
For the first integral,
note that as $r>1$, the poles are $\pm \frac{1}{r}$ and we have by
the residue theorem,
\begin{eqnarray*}
  &&\topii \oint_{|\xi|=1}\log(1+h\xi)\frac{\xi}{\xi^{2}-r^{-2}}d\xi\\
  &=&\frac{\log(1+h\xi)\cdot\xi}{\xi-r^{-1}}\Bigg|_{\xi=-r^{-1}}+\frac{\log(1+h\xi)\cdot\xi}{\xi+r^{-1}}\Bigg|_{\xi=r^{-1}}\\
  &=&\frac12 \log(1-\frac{h^{2}}{r^{2}})~.
\end{eqnarray*}
For the second integral,
\begin{eqnarray*}
  \topii \oint_{|\xi|=1}\log(1+h\xi)\frac{1}{\xi}d\xi= \log(1+h\xi)\big|_{\xi=0}=0~.
\end{eqnarray*}
The third one is
\begin{eqnarray*}
&&{\topii\oint_{|\xi|=1}\log(1+h\xi^{-1})\frac{\xi}{\xi^{2}-r^{-2}}d\xi}\\
&&=-\topii\oint_{|z|=1}\log(1+hz)\frac{z^{-1}}{z^{-2}-r^{-2}}\cdot\frac{-1}{z^{2}}dz\\
&&=\topii\oint_{|z|=1}\frac{\log(1+hz)r^{2}}{z(z+r)(z-r)}dz=\frac{\log(1+hz)r^{2}}{(z+r)(z-r)}\Bigg|_{z=0}=0~,
\end{eqnarray*}
where the first equality results from the change of variable
$z=\frac{1}{\xi}$, and the third equality holds because $r >1$,
so the only pole is $z=0$. \\
The fourth one equals
\begin{eqnarray*}
&&\topii \oint_{|\xi|=1}\log(1+h\xi^{-1})\frac{1}{\xi}d\xi=-\topii\oint_{|z|=1}\log(1+hz)\frac{-z}{z^{2}}dz\\
&&=\log(1+hz)\big|_{z=0}=0~.
\end{eqnarray*}
Collecting the four integrals leads to the desired formula for
$I_1(f,r)$.

\medskip\noindent{\em Proof of \eqref{I1g}}:\quad We have

\begin{eqnarray*}
I_1(g,r)&=& \frac{1}{2\pi i}\oint_{|\xi|=1}g(|1+h\xi|^{2})\cdot[\frac{\xi}{\xi^{2}-r^{-2}}-\frac{1}{\xi}]d\xi\\
&=& \frac{1}{2\pi i}\oint_{|\xi|=1}|1+h\xi|^{2}\cdot\Big[\frac{\xi}{\xi^{2}-r^{-2}}-\frac{1}{\xi}\Big]d\xi\\
&=& \frac{1}{2\pi i}\oint_{|\xi|=1}\frac{\xi+h+h\xi^2+h^2\xi}{\xi}\cdot\Big[\frac{\xi}{\xi^{2}-r^{-2}}-\frac{1}{\xi}\Big]d\xi\\
&=& \frac{1}{2\pi i}\oint_{|\xi|=1}\frac{\xi+h+h\xi^2+h^2\xi}{\xi^2-r^{-2}}d\xi\\
&&-\frac{1}{2\pi i}\oint_{|\xi|=1}\frac{\xi+h+h\xi^2+h^2\xi}{\xi^2}d\xi~.\\
\end{eqnarray*}

These two integrals are calculated as follows:
\begin{eqnarray*}
&&\frac{1}{2\pi i}\oint_{|\xi|=1}\frac{\xi+h+h\xi^2+h^2\xi}{\xi^2-r^{-2}}d\xi\\
&=&\frac{\xi+h+h\xi^2+h^2\xi}{\xi-r^{-1}}\Big|_{\xi=-r^{-1}}+\frac{\xi+h+h\xi^2+h^2\xi}{\xi+r^{-1}}\Big|_{\xi=r^{-1}}\\
&=&1+h^2~;
\end{eqnarray*}
\begin{eqnarray*}
&&\frac{1}{2\pi i}\oint_{|\xi|=1}\frac{\xi+h+h\xi^2+h^2\xi}{\xi^2}d\xi
=\frac{\partial}{\partial \xi}(\xi+h+h\xi^2+h^2\xi)\Big|_{\xi=0}\\
&&=1+h^2~.
\end{eqnarray*}

Collecting the two terms leads to $I_1(g,r)=0$.

\medskip\noindent{\em Proof of \eqref{I2f}}:
\begin{eqnarray*}
I_2(f)&=&\frac{1}{2\pi i}\oint_{|\xi|=1}\log(|1+h\xi|^{2})\frac{1}{\xi^{3}}d\xi\\
&=&\frac{1}{2\pi i}\Bigg[\oint_{|\xi|=1}\frac{\log(1+h\xi)}{\xi^{3}}d\xi+\oint_{|\xi|=1}\frac{\log(1+h\xi^{-1})}{\xi^{3}}d\xi\Bigg]~.\\
\end{eqnarray*}
We have
\begin{eqnarray*}
  &&\topii\oint_{|\xi|=1}\frac{\log(1+h\xi)}{\xi^{3}}d\xi=\frac12\frac{\partial^{2}}{\partial
    \xi^{2}}\log(1+h\xi)\Big|_{\xi=0} =-\frac12 h^{2}~;\\
  &&\topii\oint_{|\xi|=1}\frac{\log(1+h\xi^{-1})}{\xi^{3}}d\xi=-\topii\oint_{|z|=1}\frac{\log(1+hz)}{\frac{1}{z^{3}}}\cdot \frac{-1}{z^{2}}dz
  =0~.
\end{eqnarray*}
Combining the two leads to $I_2(f)=-\frac12 h^{2}$.

\medskip\noindent{\em Proof of \eqref{I2g}}:
\begin{eqnarray*}
&& I_2(g)
=\frac{1}{2\pi i}\oint_{|\xi|=1}\frac{(1+h\overline{\xi})(1+h\xi)}{\xi^{3}}d\xi
=\frac{1}{2\pi i}\oint_{|\xi|=1}\frac{\xi+h\xi^{2}+h+h^{2}\xi}{\xi^{4}}d\xi
=0~.
\end{eqnarray*}

\medskip\noindent{\em Proof of \eqref{J1fg}}:
\begin{eqnarray*}
J_{1}(f,g,r)
&=&\frac{1}{2\pi i}\oint_{|\xi_{2}|=1}|1+h\xi_{2}|^{2}\cdot \topii\oint_{|\xi_{1}|=1} \frac{\log(|1+h\xi_{1}|^{2})}{(\xi_{1}-r\xi_{2})^{2}}d\xi_{1}d\xi_{2}~.\\
\end{eqnarray*}
We have,
\begin{eqnarray*}
&&\topii\oint_{|\xi_{1}|=1} \frac{\log(|1+h\xi_{1}|^{2})}{(\xi_{1}-r\xi_{2})^{2}}d\xi_{1}\\
&=&\topii \oint_{|\xi_{1}|=1} \frac{\log(1+h\xi_{1})}{(\xi_{1}-r\xi_{2})^{2}}d\xi_{1}+\topii \oint_{|\xi_{1}|=1} \frac{\log(1+h\xi_{1}^{-1})}{(\xi_{1}-r\xi_{2})^{2}}d\xi_{1}~.
\end{eqnarray*}
The first term
\begin{eqnarray*}
\topii\oint_{|\xi_{1}|=1} \frac{\log(1+h\xi_{1})}{(\xi_{1}-r\xi_{2})^{2}}d\xi_{1}=0,
\end{eqnarray*}
because for fixed $|\xi_2|=1$, $|r\xi_{2}|=|r|> 1$, so $r\xi_{2}$ is
not a pole. \\
The second term is
\begin{eqnarray*}
&&\topii\oint_{|\xi_{1}|=1} \frac{\log(1+h\xi_{1}^{-1})}{(\xi_{1}-r\xi_{2})^{2}}d\xi_{1}=-\topii\oint_{|z|=1} \frac{\log(1+hz)}{(\frac{1}{z}-r\xi_{2})^{2}}\cdot \frac{-1}{z^{2}}dz\\
&&=\topii\cdot\frac{1}{(r\xi_2)^{2}}\oint_{|z|=1} \frac{\log(1+hz)}{(z-\frac{1}{r\xi_{2}})^{2}}dz=\frac{1}{(r\xi_2)^{2}}\cdot \frac{\partial}{\partial z}\log(1+hz)\Big|_{z=\frac{1}{r\xi_2}}\\
&&=\frac{h}{r\xi_2(r\xi_2+h)}~,
\end{eqnarray*}
where the first equality results from the change of variable $z=\frac{1}{\xi_1}$, and the third equality holds because for fixed $|\xi_2|=1$, $|\frac{1}{r\xi_2}|=\frac{1}{|r|}<1$, so $\frac{1}{r\xi_2}$ is a pole of second order.\\
Therefore,
\begin{eqnarray*}
&&J_1(f,g,r)\\
&=&\frac{h}{2\pi ir^{2}}\oint_{|\xi_{2}|=1}\frac{(1+h\xi_2)(1+h\overline{\xi_2})}{\xi_{2}(\xi_{2}+\frac{h}{r})}d\xi_{2}\\
&=&\frac{h}{2\pi ir^{2}}\oint_{|\xi_{2}|=1}\frac{\xi_2+h\xi_2^{2}+h+h^{2}\xi_2}{\xi_{2}^{2}(\xi_{2}+\frac{h}{r})}d\xi_{2}\\
&=&\frac{h}{2\pi ir^{2}}\Bigg[\oint_{|\xi_{2}|=1}\frac{1+h^2}{\xi_{2}(\xi_{2}+\frac{h}{r})}d\xi_{2}+
\oint_{|\xi_{2}|=1}\frac{h}{\xi_{2}+\frac{h}{r}}d\xi_{2}
+\oint_{|\xi_{2}|=1}\frac{h}{\xi_{2}^2(\xi_{2}+\frac{h}{r})}d\xi_{2}\Bigg]~.
\end{eqnarray*}
Finally we find $J_1(f,g,r)=\frac{h^2}{r^2}$ since
\begin{eqnarray*}
&&\frac{h}{2\pi
    ir^{2}}\oint_{|\xi_{2}|=1}\frac{1+h^2}{\xi_{2}(\xi_{2}+\frac{h}{r})}d\xi_{2}=0~,\quad
  \frac{h}{2\pi ir^{2}}\oint_{|\xi_{2}|=1}\frac{h}{\xi_{2}+\frac{h}{r}}d\xi_{2}=\frac{h^2}{r^2}~,\\
&&\frac{h}{2\pi ir^{2}}\oint_{|\xi_{2}|=1}\frac{h}{\xi_{2}^2(\xi_{2}+\frac{h}{r})}d\xi_{2}=0~.
\end{eqnarray*}

\medskip\noindent{\em Proof of \eqref{J1ff}}:
\begin{eqnarray*}
J_1(f,f,r)&=&\topii\oint_{|\xi_{2}|=1}f(|1+h\xi_{2}|^{2})\cdot \topii\oint_{|\xi_{1}|=1} \frac{f(|1+h\xi_{1}|^{2})}{(\xi_{1}-r\xi_{2})^{2}}d\xi_{1}d\xi_{2}\\
&=& \topii\oint_{|\xi_{2}|=1}f(|1+h\xi_{2}|^{2})\frac{h}{r\xi_2(r\xi_2+h)}d\xi_{2}\\
&=&\frac{h}{2\pi ir^{2}}\oint_{|\xi_{2}|=1}\frac{\log(1+h\xi_{2})}{\xi_{2}(\frac{h}{r}+\xi_{2})}d\xi_{2}+\frac{h}{2\pi ir^{2}}\oint_{|\xi_{2}|=1}\frac{\log(1+h\xi_{2}^{-1})}{\xi_{2}(\frac{h}{r}+\xi_{2})}d\xi_{2}~.\\
\end{eqnarray*}
We have
\begin{eqnarray*}
&&\frac{h}{2\pi ir^{2}}\oint_{|\xi_{2}|=1}\frac{\log(1+h\xi_{2})}{\xi_{2}(\frac{h}{r}+\xi_{2})}d\xi_{2}\\
&=&\frac{h}{r^{2}}\Bigg[\frac{\log(1+h\xi_{2})}{\frac{h}{r}+\xi_{2}}\Bigg|_{\xi_{2}=0}+\frac{\log(1+h\xi_{2})}{\xi_{2}}\Bigg|_{\xi_{2}=-\frac{h}{r}}\Bigg]\\
&=&-\frac{1}{r}\log(1-\frac{h^{2}}{r})~,\
\end{eqnarray*}
and
\begin{eqnarray*}
&&\frac{h}{2\pi ir^{2}}\oint_{|\xi_{2}|=1}\frac{\log(1+h\xi_{2}^{-1})}{\xi_{2}(\frac{h}{r}+\xi_{2})}d\xi_{2}=\frac{-h}{2\pi ir^{2}}\oint_{|z|=1}\frac{\log(1+hz)}{\frac{1}{z}(\frac{h}{r}+\frac{1}{z})}\cdot \frac{-1}{z^2}dz\\
&&=\frac{1}{2\pi ir}\oint_{|z|=1}\frac{\log(1+hz)}{z+\frac{r}{h}}dz
=0~,
\end{eqnarray*}
where the first equality results from the change of variable $z=\frac{1}{\xi_2}$, and the third equality holds because $|\frac{r}{h}|> 1$, so $\frac{r}{h}$ is not a pole.

\noindent Finally, we find $J_1(f,f,r)=-\frac{1}{r}\log(1-\frac{h^{2}}{r})$~.

\medskip\noindent{\em Proof of \eqref{J1gg}}:
\begin{eqnarray*}
J_1(g,g,r)&=&\topii\oint_{|\xi_{2}|=1}|1+h\xi_{2}|^{2}\cdot\topii\oint_{|\xi_{1}|=1}\frac{|1+h\xi_{1}|^{2}}{(\xi_{1}-r\xi_{2})^{2}}d\xi_{1}d\xi_{2}~.
\end{eqnarray*}
We have
\begin{eqnarray*}
&&\topii\oint_{|\xi_{1}|=1}\frac{|1+h\xi_{1}|^{2}}{(\xi_{1}-r\xi_{2})^{2}}d\xi_{1}
=\topii\oint_{|\xi_{1}|=1}\frac{\xi_1+h\xi_1^{2}+h+h^{2}\xi_1}{\xi_1(\xi_{1}-r\xi_{2})^{2}}d\xi_{1}\\
&&=\topii\Bigg[\oint_{|\xi_{1}|=1}\frac{1+h^2}{(\xi_{1}-r\xi_{2})^{2}}d\xi_{1}+\oint_{|\xi_{1}|=1}\frac{h\xi_1}{(\xi_{1}-r\xi_{2})^{2}}d\xi_{1}\\
&&+\oint_{|\xi_{1}|=1}\frac{h}{\xi_1(\xi_{1}-r\xi_{2})^{2}}d\xi_{1}\Bigg]\\
&&=\frac{h}{r^2\xi_2^{2}}~,
\end{eqnarray*}
since
\begin{eqnarray*}
&&\topii\oint_{|\xi_{1}|=1}\frac{1+h^2}{(\xi_{1}-r\xi_{2})^{2}}d\xi_{1}=0~,\quad
\topii\oint_{|\xi_{1}|=1}\frac{h\xi_1}{(\xi_{1}-r\xi_{2})^{2}}d\xi_{1}=0~,\\
&&\topii\oint_{|\xi_{1}|=1}\frac{h}{\xi_1(\xi_{1}-r\xi_{2})^{2}}d\xi_{1}=\frac{h}{(\xi_{1}-r\xi_{2})^{2}}\Bigg|_{\xi_1=0}=\frac{h}{r^2\xi_2^{2}}~,\\
\end{eqnarray*}
where the equality above holds because
for fixed $|\xi_2|=1$, $|r\xi_{2}|=|r|> 1$, so $r\xi_{2}$ is not a
pole. Therefore,
\begin{eqnarray*}
  J_1(g,g,r)
&=&\frac{h}{2\pi ir^{2}}\oint_{|\xi_{2}|=1}\frac{\xi_2+h\xi_2^{2}+h+h^{2}\xi_2}{\xi_{2}^{3}}d\xi_{2}\\
&=&\frac{h}{2\pi
    ir^{2}}\Bigg[\oint_{|\xi_{2}|=1}\frac{1+h^2}{\xi_{2}^{2}}d\xi_{2}+\oint_{|\xi_{2}|=1}\frac{h}{\xi_{2}}d\xi_{2}+\oint_{|\xi_{2}|=1}\frac{h}{\xi_{2}^{3}}d\xi_{2}\Bigg]~,\\
  &=&\frac{h^2}{r^2}~.
\end{eqnarray*}

\medskip\noindent{\em Proof of \eqref{J2fg}~\eqref{J2ff}~\eqref{J2gg}}:
We have
\begin{eqnarray*}
&&\topii\oint_{|\xi_{1}|=1}\frac{f(|1+h\xi_{1}|^{2})}{\xi_{1}^{2}}d\xi_{1}
=\topii\oint_{|\xi_{1}|=1}\frac{\log(|1+h\xi_{1}|^{2})}{\xi_{1}^{2}}d\xi_{1}\\
&&=\topii\oint_{|\xi_{1}|=1}\frac{\log(1+h\xi_{1})+\log(1+h\xi_{1}^{-1})}{\xi_{1}^{2}}d\xi_{1}~=h~,
\end{eqnarray*}
since
\begin{eqnarray*}
  &&\topii\oint_{|\xi_{1}|=1}\frac{\log(1+h\xi_{1})}{\xi_{1}^{2}}d\xi_{1}= \frac{\partial}{\partial\xi_{1}}\Big(\log(1+h\xi_{1})\Big)\Bigg|_{\xi_{1}=0}=h~,\\
  &&\topii\oint_{|\xi_{1}|=1}\frac{\log(1+h\xi_{1}^{-1})}{\xi_{1}^{2}}d\xi_{1}=-\topii\oint_{|z|=1}\frac{\log(1+hz)}{\frac{1}{z^{2}}}\cdot (-\frac{1}{z^{2}}dz)\\
&&=\topii\oint_{|z|=1}\log(1+hz)dz
=0~.
\end{eqnarray*}
Similarly,
\begin{eqnarray*}
&&\topii\oint_{|\xi_{2}|=1}\frac{g(|1+h\xi_{2}|^{2})}{\xi_{2}^{2}}d\xi_{2}
=\topii\oint_{|\xi_{2}|=1}\frac{\xi_2+h\xi_2^{2}+h+h^2\xi_2}{\xi_{2}^{3}}d\xi_{2}
=h.
\end{eqnarray*}
Therefore,
\begin{eqnarray*}
&&J_2(f,g)=\topii\oint_{|\xi_{1}|=1}\frac{f(|1+h\xi_{1}|^{2})}{\xi_{1}^{2}}d\xi_{1}\cdot \topii\oint_{|\xi_{2}|=1}\frac{g(|1+h\xi_{2}|^{2})}{\xi_{2}^{2}}d\xi_{2}=h^2~,\\
&&J_2(f,f)=\topii\oint_{|\xi_{1}|=1}\frac{f(|1+h\xi_{1}|^{2})}{\xi_{1}^{2}}d\xi_{1}\cdot \topii\oint_{|\xi_{2}|=1}\frac{f(|1+h\xi_{2}|^{2})}{\xi_{2}^{2}}d\xi_{2}=h^2~,\\
&&J_2(g,g)=\topii\oint_{|\xi_{1}|=1}\frac{g(|1+h\xi_{1}|^{2})}{\xi_{1}^{2}}d\xi_{1}\cdot \topii\oint_{|\xi_{2}|=1}\frac{g(|1+h\xi_{2}|^{2})}{\xi_{2}^{2}}d\xi_{2}=h^2~.\\
\end{eqnarray*}

\noindent The proof of Lemma \ref{lem:joint} is complete.\\

\subsection{Proof of Lemma~\protect\ref{cnlemma:joint}}

Let $f(x)=x^2$ and $g(x)=x$.
Define $C_n$ and $B_n$  by the decompositions
\begin{eqnarray*}
  \sum_{i=1}^{p} \ell_i ^2& = & p\int f(x)
  d(F_{n}(x)-F^{y_{n}}(x))+pF^{y_{n}}(f) = C_n+pF^{y_{n}}(f)~ , \\
  \sum_{i=1}^{p}\ell_i & = &
  p\int g(x) d(F_{n}(x)-F^{y_{n}}(x))+pF^{y_{n}}(g)= B_n + pF^{y_{n}}(g)  ~.
\end{eqnarray*}
Applying
\reu{Theorem~\ref{t1}}  given in the Appendix  to the pair $(f,g)$,
we have
\[
\left(
\begin{array}{cc}
  C_n  \\  B_n
\end{array}
\right)
\Longrightarrow
N\Bigg(
\left(
\begin{array}{cc}
  EX_{f}  \\  EX_{g}
\end{array}
\right), \
\left(
\begin{array}{cc}
  \cov(X_f,X_f)
  &
  \cov(X_f, X_g)
  \\
  \cov(X_g,X_f)
  & \cov(X_g,X_g)
\end{array}
\right) \
\Bigg).
\]
It remains to evaluate the limiting parameters
and this results from the following calculations:
\begin{eqnarray}
  I_1(f,r)&=& \frac{h^2}{r^2}~,\label{2I1f}\\
  I_1(g,r) &=& 0 ~,                               \label{2I1g} \\
  I_2(f)&=&  h^2~,        \label{2I2f} \\
  I_2(g)&=& 0~,                  \label{2I2g} \\
  J_1(f,g,r) &=& \frac {2h^2+2h^4}{r^2}~,  \label{2J1fg}\\
  J_1(f,f,r) &=& \frac{2h^4+(2h+2h^3)^2r}{r^3} ~,  \label{2J1ff}\\
  J_1(g,g,r) &=& \frac {h^2}{r^2}~,  \label{2J1gg}\\
  J_2(f,g) &=&  {2h^2+2h^4} ~,  \label{2J2fg}\\
  J_2(f,f) &=&  {(2h+2h^3)^2} ~,  \label{2J2ff}\\
  J_2(g,g) &=&  {h^2} ~.  \label{2J2gg}
\end{eqnarray}

\noindent The results \eqref{2I1g}, \eqref{2I2g}, \eqref{2J1gg} and
\eqref{2J2gg} are
 exactly the same as  those found
 in \eqref{I1g}, \eqref{I2g}, \eqref{J1gg} and
\eqref{J2gg} in the proof of
Lemma~\ref{lem:joint}. The remaining results are found by similar
calculations using Proposition~\ref{t2} in the Appendix
and their details are omitted.

\section{Concluding remarks}
Using recent central limit theorems for eigenvalues of large sample covariance matrices, we are able to find new asymptotic distributions for two major procedures to test the sphericity of a large-dimensional distribution. Although the theory is developed under the scheme $p\rightarrow \infty$, $n\rightarrow \infty$ and $p/n \rightarrow y >0$, our Monte-Carlo study has proved that: on the one hand, both CLRT and CJ are already very efficient for middle dimension such as $(p,n)=(96,128)$ both in size and power, see Table \ref{SIZE} and Table \ref{POWER}; and on the other hand, CJ also behaves very well in most of ``large $p$, small $n$'' situation, see Table \ref{cnsizepower}.

Three characteristic features emerge from our findings:
\begin{enumerate}
\item[(a)]
  These asymptotic distributions are universal in the sense that they depend on the distribution of the observations only through its first four moments;
\item[(b)]
  The new test procedures improve quickly when either the dimension
  $p$ or the sample size $n$ gets large. In particular, for a given
  sample size $n$, within a wide range of values of $p/n$, higher dimensions $p$ lead to better performance of
  these corrected test statistics.
\item[(c)]
  CJ is particularly robust against the
  dimension inflation. Our Monte-Carlo study shows that for a small
  sample size $n=64$, the test is effective for $0<p/n\le 20$.
\end{enumerate}

In a sense, these new procedures have benefited from the ``blessings of dimensionality''.

\section*{Acknowledgements}
We are grateful to the associate editor and two referees for their numerous comments that have lead to important modifications of the paper.


%

\appendix
\section{Formula for limiting parameters in the CLT
  for eigenvalues of a sample covariance matrix
  with general fourth moments}

Given a sample covariance matrix $S_n$ of dimension $p$
with eigenvalues $\lambda_1,\ldots,\lambda_p$,  linear spectral
statistics
of the form  $F_n(g)=p^{-1} \sum_{i=1}^{p} g(\lambda_j)$ for suitable functions $g$
are of central importance  in multivariate analysis.
Such CLT's have been successively developed  since the pioneering work of
\cite{r12}, see  \cite{r2} and  \cite{r16}
for a recent account on the subject.

The CLT in \cite{r2} (see also an improved  version in
\cite{r5})  has been widely used in applications
as  this CLT also provides, for the first time,
explicit formula for  the mean and
covariance parameters of the normal limiting  distribution.
In the special case with an  array $\{x_{ij}\}$ of independent variables,
this CLT  assumes the following moment conditions:
\begin{enumerate}
\item[(a)]
  For each $n$,
  $x_{ij}=x_{ij}^{n},i\leqslant p,j\leqslant n$ are independent.
\item[(b)]
  $Ex_{ij}=0, E|x_{ij}|^{2}=1, \max_{i,j,n}E|x_{ij}|^{4} <\infty$.
\item[(c)] If $\{x_{ij}\}$'s   are real, then $Ex_{ij}^{4}=3$;
  otherwise (complex variables), $Ex^{2}_{ij}=0$ and
  $E|x_{ij}|^{4}=2$.
\end{enumerate}

In Condition (c), the fourth moments of the
entries are set to the values 3 or 2  matching the normal case.
This is indeed a quite demanding and restrictive condition since in
the real case for example, it is
incredibly hard to find  a non-normal
distribution with mean 0, variance 1 and
fourth moment equaling 3. As a consequence, most of
if not all applications
published in the literature using this CLT assumes a normal
distribution for the  observations.
Recently, effort have been made in
\cite{r21}, \cite{r16}  and \cite{r29}
to overcome these moment restrictions.
We present below such a CLT  with general forth moments that will be
used for the sphericity test.

In all the following, we use an indicator $\kappa$ set  to 2
when $\{x_{ij}\}$ are real and to 1 when they are complex.
Define  $\beta=E|x_{ij}|^{4}-1-\kappa$ for both cases and
$h=\sqrt y$.

For the presentation of the results, let be the sample covariance
matrix $S_n =\frac1n \sum_{i=1}^n X_i X_i^*$  where $X_i=
(x_{ki})_{1\le k\le p}$ is the $i$-th observed  vector.
It is then well-known that when
$p\to\infty$, $n\to\infty$ and  $p/n\to y >0~,$
the distribution of its eigenvalues
converges to a nonrandom  distribution, namely the Mar\v{c}enko-Pastur
distribution $F^y$ with support
$[a,b]=[(1\pm\sqrt y)^2]$ (an additional mass at the origin when
$y>1$).
 Moreover, the Stieltjes transform
$\mbar$ of  a companion distribution defined by
$\underline F^y = (1-y)\delta_0+yF_c$  satisfies
an  inverse equation for $z\in \mathbb{C}^+$,
\begin{equation}\label{eq:z}
  z = -\frac1\mbar + \frac{y}{1+\mbar}.
\end{equation}

The following CLT is a particular instance
of Theorem 1.4  in  \cite{r21}.
\begin{theorem}\label{t1}[\cite{r21}]
  \quad Assume that  for each $n$,
  the variables $x_{ij}=x_{ij}^{n},i\leqslant p,j\leqslant n$
  are independent and identically distributed satisfying
  $Ex_{ij}=0$, $E|x_{ij}|^{2}=1$,
  $E|x_{ij}|^{4}=\beta+1+\kappa <\infty$ and
  in case of they are complex, $Ex^{2}_{ij}=0$.
  Assume moreover,
  \[   p\to\infty, ~~n\to\infty, ~~ p/n\to y >0~. \]
  Let $f_{1},\cdots f_{k}$ be functions analytic on an open region
  containing the support of $F^y$.
  The random vector $\{ X_{n}(f_{1}),\cdots X_{n}(f_{k})\}$ where
  \[  X_n(f) = p\left\{  F_n (f) - F^{y_n}(f) \right\}
  \]
  converges
  weakly to a normal vector $(X_{f_{1}},\cdots X_{f_{k}})$
  with mean function and
  covariance function:
  \begin{eqnarray}
    \E[X_f]& =&   (\kappa-1)  I_1(f) + \beta  I_2(f) ~,
    \label{eq:E}\\
    \cov(X_{f},X_{g}) &=& \kappa J_1(f,g) + \beta
    J_2(f,g)~,
    \label{eq:cov}
  \end{eqnarray}
  where
  \begin{eqnarray*}
  I_1(f) &=& - \topii \oint
  \frac{y\left\{\mbar/(1+\mbar)\right\}^3(z)f(z)}
       { \left[1-y \left\{ \mbar/(1+\mbar)\right\}^2 \right]^2}dz~,
  \\
  I_2(f) &=&
  -\frac{1}{2\pi i}\oint
  \frac{y\left\{\mbar/(1+\mbar)\right\}^3(z)f(z)}
       { 1-y \left\{ \mbar/(1+\mbar)\right\}^2 }dz~,       \\
  J_1(f,g) &=&
  -\frac{1}{4\pi^{2}}
  \oint\oint\frac{f(z_{1})g(z_{2})}{(\underline{m}(z_{1})-\underline{m}(z_{2}))^{2}}\underline{m}'(z_{1})\underline{m}'(z_{2})dz_{1}dz_{2}~,\\
  J_2(f,g) &=& \frac{-y}{4\pi^2}
  \oint f(z_1) \frac{\partial}{\partial z_1}
  \left\{ \frac{\mbar}{1+\mbar}  (z_1)\right\} dz_1
  \cdot
  \oint g(z_2) \frac{\partial}{\partial z_2}
  \left\{ \frac{\mbar}{1+\mbar}  (z_2)\right\} dz_2~,
  \end{eqnarray*}
   where the integrals are along contours (non overlapping in $J_1$) enclosing
the support of $F^y$.
\end{theorem}

However, concrete applications of this CLT are not easy
since  the limiting parameters are given through
those integrals on contours that are only vaguely defined.
The purpose of this appendix is to go a step further
by providing alternative formula for these limiting parameters.
These new  formulas, presented in the following Proposition
convert all these integral along the unit circle; they are
much easier to use  for concrete applications, see for instance
the proofs of  Lemma~\ref{lem:joint}
and \ref{cnlemma:joint} in the paper.
Furthermore, these formulas will be of independent interest
for applications other than those developed in this paper.

\begin{proposition}\label{t2}
  The limiting parameters in Theorem~\ref{t1} can be expressed as
  following:
  \begin{eqnarray}
    I_1(f) &=& \lim_{r\downarrow 1}  I_1(f,r)~,\label{I1}\\
    I_2(f) &=& \frac{1}{2\pi i}\oint_{|\xi|=1}f(|1+h\xi|^{2})\frac{1}{\xi^{3}}d\xi~,\label{I2}\\
    J_1(f,g) &=& \lim_{r\downarrow 1}  J_1(f,g,r)~, \label{J1}\\
    J_2(f,g)
   &=&-\frac{1}{4\pi^{2}}\oint_{|\xi_{1}|=1}\frac{f(|1+h\xi_{1}|^{2})}{\xi_{1}^{2}}d\xi_{1}\oint_{|\xi_{2}|=1}\frac{g(|1+h\xi_{2}|^{2})}{\xi_{2}^{2}}d\xi_{2}~,\label{J2}
  \end{eqnarray}
  with
  \begin{eqnarray*}
    I_1(f,r) &=& \frac1{2\pi i}\oint_{|\xi|=1}f(|1+h\xi|^{2})[\frac{\xi}{\xi^{2}-r^{-2}}-\frac{1}{\xi}]d\xi
    ~,\\
    J_1(f,g,r) &=&
    -\frac{1}{4\pi^{2}}\oint_{|\xi_{1}|=1}\oint_{|\xi_{2}|=1}
    \frac{f(|1+h\xi_{1}|^{2})g(|1+h\xi_{2}|^{2})} {(\xi_{1}-r\xi_{2})^{2}}d\xi_{1}d\xi_{2}~.
  \end{eqnarray*}
\end{proposition}

\begin{proof}
  We start with the simplest formula $I_2(f)$ to explain the main
  argument
  and indeed, the other formula are obtained similarly.
  The idea is to introduce the change of the variable
  $z=1+hr\xi+h r^{-1}\overline\xi+h^2$ with $r>1$ but close to 1 and
  $|\xi|=1$
  (recall $h=\sqrt y$). Note that this idea has been employed in
  \cite{r29}.
  It can be readily checked that when $\xi$ runs counterclockwisely
  the unit circle, $z$ will run a contour $\cal C$
  that encloses  closely
  the support interval $[a,b]=[(1\pm h)^2]$ (recall $h=\sqrt y$).
  Moreover, by the Eq.~\eqref{eq:z}, we have on $\cal C$
  \[     \mbar = -\frac{1}{1+hr\xi},
  \quad \text{and~~} dz = h(r-r^{-1}\xi^{-2})d\xi~.
  \]
  Applying this variable change to the formula of  $I_2(f)$ given in
  Theorem~\ref{t1}, we have
  \begin{eqnarray*}
    I_2(f) &=& \lim_{r\downarrow 1}
    \topii \oint_{|\xi|=1} f(z) \frac{1}{\xi^3}\frac{r\xi^2-r^{-1}}{r(r^2\xi^2-1)}d\xi\\
    &=& \topii   \oint_{|\xi|=1} f(|1+h\xi|^2) \frac{1}{\xi^3}d\xi~.
  \end{eqnarray*}
  This proves the formula~\eqref{I2}. For \eqref{I1}, we have similarly
  \begin{eqnarray*}
      I_1(f) &=& \lim_{r\downarrow 1}
    \topii \oint_{|\xi|=1} f(z)
    \frac{1}{\xi^3}\frac{r\xi^2-r^{-1}}{r(r^2\xi^2-1)}  \frac1{1-r^{-2}\xi^{-2}}d\xi\\
    &=& \lim_{r\downarrow 1}
    \topii \oint_{|\xi|=1}  f(|1+h\xi|^2) \frac1{\xi(\xi^2-r^{-2})}\\
    &=& \lim_{r\downarrow 1} I_1(f,r)~.
  \end{eqnarray*}
  Formula~\eqref{J2} for $J_2(f,g)$
  is calculated in a same fashion by observing
  that we have
  \[
  \frac{\partial}{\partial z}
  \left\{ \frac{\mbar}{1+\mbar}  (z)\right\} dz
  =   \frac{\partial}{\partial \xi}
  \left\{ \frac{\mbar}{1+\mbar}  (\xi)\right\} d\xi
  =  \frac{\partial}{\partial \xi}
  \left\{ \frac{1}{-hr\xi} \right\} d\xi
  = \frac{1}{hr\xi^2}d\xi~.
  \]
  Finally for \eqref{J1}, we use two non-overlapping contours
  defined by $z_j= 1+hr_j\xi_j+h r_j^{-1}\overline\xi_j+h^2       $,
  $j=1,2$ where $r_2>r_1>1$.  By observing that
  \[  \mbar'(z_j)dz_j = \left(\frac{\partial}{\partial
    \xi_j}\mbar\right)d\xi_j =
  \frac{hr_j}{(1+hr_j\xi_j)^2} d\xi_j~,
  \]
  we find
  \begin{eqnarray*}
    J_1(f,g) &=&
    \lim_{ {\footnotesize
      \begin{array}{c}r_2>r_1>1 \\[-1.5mm]  r_2\downarrow 1
      \end{array}}
    }
    -\frac1{4\pi^2}
    \oint_{|\xi_1|=1}\!\oint_{|\xi_2|=1}
    \frac{f(z_1)g(z_2)} { \left\{ \mbar(z_1)-\mbar(z_2) \right\}^2 }\\
    &&\cdot \frac{hr_1}{(1+hr_1\xi_1)^2}
    \cdot \frac{hr_2}{(1+hr_2\xi_2)^2}  d\xi_1d\xi_2 \\
    &=&
    \lim_{ {\footnotesize
        \begin{array}{c}r_2>r_1>1, \\[-1.5mm]  r_2\downarrow 1
      \end{array}}
    }
    -\frac1{4\pi^2}
    \oint_{|\xi_1|=1}\!\oint_{|\xi_2|=1}
    \frac{f(z_1)g(z_2)}{\left\{ r_1\xi_1-r_2\xi_2\right\}^2 } d\xi_1d\xi_2
    \\
    &=&   \lim_{r\downarrow 1}
    -\frac1{4\pi^2}
    \oint_{|\xi_1|=1}\!\oint_{|\xi_2|=1}
    \frac{f(|1+h\xi_1|^2)g(|1+h\xi_2|^2)}{\left\{ \xi_1-r\xi_2\right\}^2 } d\xi_1d\xi_2~.
  \end{eqnarray*}
  The proof is complete.
\end{proof}

\end{document}